\documentclass[12pt,reqno]{amsart}
\usepackage{amssymb,amsmath,amsthm,enumerate,bbm,color,verbatim}
\usepackage[a4paper]{geometry}
\usepackage{graphicx} 

\makeatletter
\@namedef{subjclassname@2020}{\textup{2020} Mathematics Subject Classification} 
\makeatother

\DeclareMathOperator{\Ran}{Ran}

\DeclareMathOperator{\Ker}{Ker}

\DeclareMathOperator{\supp}{supp}

\DeclareMathOperator{\Span}{span}

\renewcommand{\Re}{\operatorname{Re}}
\renewcommand{\Im}{\operatorname{Im}}

\newcommand{\abs}[1]{\lvert#1\rvert}

\newcommand{\norm}[1]{\lVert#1\rVert}

\newcommand{\jap}[1]{\langle#1\rangle}
\newcommand{\Jap}[1]{\left\langle#1\right\rangle}

\newcommand{\bbR}{{\mathbb R}}
\newcommand{\bbC}{{\mathbb C}}
\newcommand{\bbN}{{\mathbb N}}

\newcommand{\bbD}{{\mathbb D}}
\newcommand{\bbP}{{\mathbb P}}

\newcommand{\bP}{\mathbf{P}}
\newcommand{\bJ}{\mathbf{J}}
\newcommand{\bbJ}{\mathbb{J}}

\newcommand{\calH}{{\mathcal H}}

\newcommand{\calP}{\mathcal{P}}

\numberwithin{equation}{section}

\theoremstyle{plain}
\newtheorem{theorem}{\bf Theorem}[section]
\newtheorem*{theorem*}{Theorem 1.1$'$}
\newtheorem{lemma}[theorem]{\bf Lemma}
\newtheorem*{lemma*}{\bf Lemma}
\newtheorem{proposition}[theorem]{\bf Proposition}

\theoremstyle{definition}

\theoremstyle{remark}
\newtheorem*{remark*}{\bf Remark}
\newtheorem{remark}[theorem]{\bf Remark}

\newcommand{\ii}{\mathrm{i}} 
\newcommand{\dd}{\mathrm{d}} 

\newcommand{\meas}{{\mathbf M}}

\usepackage[normalem]{ulem} 

\begin{document}

\title[Inverse spectral problem for non-self-adjoint Jacobi matrices]{An inverse spectral problem for non-self-adjoint Jacobi matrices}

\author{Alexander Pushnitski}
\address{Department of Mathematics, King's College London, Strand, London, WC2R~2LS, U.K.}
\email{alexander.pushnitski@kcl.ac.uk}

\author{Franti\v{s}ek \v{S}tampach}
\address{Department of Mathematics, Faculty of Nuclear Sciences and Physical Engineering, 
Czech Technical University in Prague, Trojanova 13, 12000 Prague~2, Czech Republic.}
\email{frantisek.stampach@cvut.cz}



\date{\today}

\begin{abstract}
We consider the class of bounded symmetric Jacobi matrices $J$ with positive off-diagonal elements and complex diagonal elements. With each matrix $J$ from this class, we associate the spectral data, which consists of a pair $(\nu,\psi)$. Here $\nu$ is the spectral measure of $|J|=\sqrt{J^*J}$ and  $\psi$ is a \emph{phase function} on the real line satisfying $|\psi|\leq1$ almost everywhere with respect to the measure $\nu$. Our main result is that the map from $J$ to the pair $(\nu,\psi)$ is a bijection between our class of Jacobi matrices and the set of all spectral data. 
\end{abstract}

\maketitle

\section{Overview and Background} 

\subsection{Overview}
A \emph{Jacobi matrix} is an infinite tri-diagonal matrix of the form
\begin{equation}
J=\begin{pmatrix}
b_0 & a_0 & 0 & 0 & 0 &\cdots\\
a_0 & b_1 & a_1 & 0 & 0 & \cdots\\
0 & a_1 & b_2 & a_2 & 0 &\cdots\\
0 & 0 & a_2 & b_3 & a_3 & \cdots\\
\vdots&\vdots&\vdots&\vdots&\vdots&\ddots
\end{pmatrix},
\label{a.jacobi}
\end{equation}
where $a_j\not=0$ and $b_j$ are  real or complex numbers, known as the \emph{Jacobi parameters}. Throughout the paper, we assume that $a_j$ and $b_j$ are bounded:
\begin{equation}
\sup_{j\geq0}(\abs{a_j}+\abs{b_j})<\infty;
\label{a1}
\end{equation}
this is equivalent to the boundedness of $J$ as a linear operator on the Hilbert space $\ell^2(\bbN_0)$, $\bbN_0=\{0,1,2,\dots\}$. 

The vast majority of the literature on Jacobi matrices is devoted to the case of self-adjoint $J$: 
\begin{equation}
a_j>0, \quad b_j\in\bbR, \quad \forall j\geq0.
\label{a.sa}
\end{equation}
Here $a_j\not=0$ is essential, while $a_j>0$ is a normalisation condition and can be replaced, for example, by $a_j<0$ for all $j$ or by $a_j$ having any prescribed sequence of signs (see Remark~\ref{rmk.a3} below). Let $\mu$ be the spectral measure of $J$ corresponding to the first vector $\delta_0=(1,0,0,\dots)$ of the standard basis in $\ell^2(\bbN_0)$ (see \eqref{eq.spmeasure} below). 
The inverse spectral problem for bounded self-adjoint Jacobi matrices can be succinctly stated as follows: \emph{there is a bijection between all bounded Jacobi matrices $J$ satisfying \eqref{a.sa} and all non-degenerate probability measures  $\mu$ with bounded support.} Here \emph{non-degenerate} means having infinite support, i.e. not reducing to a finite linear combination of point masses. We recall this bijection and its context in Section~\ref{sec.sa} below. 

The main result of this paper is an extension of this inverse spectral problem to the case of bounded non-selfadjoint Jacobi matrices satisfying 
\begin{align}
a_j>0, \quad b_j\in\bbC, \quad \forall j\geq0.
\label{a2}
\end{align}
Here again $a_j>0$ can be replaced by $a_j$ having any prescribed sequence of complex arguments (see Remark~\ref{rmk.arg}).
Denote $\abs{J}=\sqrt{J^*J}$ and let $\nu$ be the spectral measure of $\abs{J}$ corresponding to the vector $\delta_0$. We prove that there is a bijection between all bounded Jacobi matrices $J$ satisfying \eqref{a2} and all pairs $(\nu,\psi)$, where $\nu$ is a non-degenerate probability measure on $[0,\infty)$ and $\psi\in L^\infty(\nu)$ is a certain auxiliary function, which we call a \emph{phase function}, satisfying the constraint $\abs{\psi(s)}\leq 1$ for $\nu$-a.e. $s>0$. 
In Section~\ref{sec.main} we introduce the \emph{spectral data} $(\nu,\psi)$ and state precisely our main result and its consequences. Proofs are given in Sections~\ref{sec.b}--\ref{sec.e}. In Section~\ref{sec.example} we consider an example.

\subsection{Bounded self-adjoint Jacobi matrices}\label{sec.sa}
To set the scene, here we recall classical results on direct and inverse spectral problem for bounded \emph{self-adjoint} Jacobi matrices. For the details, we refer e.g. to \cite{akh_21} or to \cite{ges-sim_jam97} and the literature cited therein.

Under the assumptions \eqref{a1} and \eqref{a.sa}, $J$ is a bounded self-adjoint operator on $\ell^2(\bbN_0)$. The spectral theory of $J$ is intimately connected to the theory of orthogonal polynomials.
The usual approach in the theory of orthogonal polynomials is to fix a finite measure $\mu$ on the real line and to apply the Gram-Schmidt procedure to the sequence of powers $1,x,x^2,\dots$ with respect to the metric of $L^2(\mu)$. This produces the sequence of normalised orthogonal polynomials $p_{j}$.  The Jacobi matrix $J$ appears as the matrix of coefficients in the three-term recurrence relation
\begin{equation}
\begin{aligned}
 b_{0}p_{0}(x)+a_{0}p_{1}(x)&=xp_{0}(x),\\
 a_{j-1}p_{j-1}(x)+b_{j}p_{j}(x)+a_{j}p_{j+1}(x)&=xp_{j}(x), \quad j\geq1,
\end{aligned}
\label{a.ogpol}
\end{equation}
with $p_{0}=1$. This means that one starts with $\mu$ and determines $J$, i.e. $\mu\mapsto J$. In order to motivate what comes next, we need to take a slightly different view on the subject, starting from $J$ as an operator on $\ell^2(\bbN_0)$ and then coming to $\mu$ as the spectral measure of $J$, i.e. $J\mapsto \mu$. Below we recall the key statements of the theory presented in this light. 

We denote by $\delta_j$, $j\geq0$, the vectors of the standard basis in $\ell^2(\bbN_0)$ and 
$\jap{\cdot,\cdot}$ the standard inner product in $\ell^2(\bbN_0)$ linear in the first argument and anti-linear in the second one. The \emph{spectral measure} $\mu$ of a bounded self-adjoint Jacobi matrix $J$ is the probability measure 
\begin{equation}
 \mu(\Delta)=\jap{\chi_{\Delta}(J)\delta_{0},\delta_{0}},
 \label{eq.spmeasure}
\end{equation}
where $\Delta\subset\bbR$ is a Borel set and $\chi_{\Delta}$ the indicator function of $\Delta$.
Since $\mu$ is compactly supported in $\bbR$ (on the spectrum of $J$), it is uniquely determined by its moments
\begin{equation}
\jap{J^m\delta_0,\delta_0}=\int_{-\infty}^\infty s^m\dd\mu(s), \quad m\geq0.
\label{a.moments}
\end{equation}

Recall that a vector $w$ is said to be a \emph{cyclic vector} for a bounded self-adjoint operator $T$ on a~Hilbert space if the set of finite linear combinations of $T^{m}w$, $m\geq0$, is dense. If a cyclic vector for $T$ exists, the spectrum of $T$ is said to be \emph{simple}.
The \emph{direct spectral problem} for self-adjoint Jacobi matrices can be stated as follows:

\begin{theorem}[Direct spectral problem, self-adjoint case]\label{thm.0a}
Let $J$ be a bounded self-adjoint Jacobi matrix. Then $\delta_0$ is a cyclic vector for $J$. Consequently, $J$ is unitarily equivalent to the operator of multiplication by the independent variable in $L^2(\mu)$. 
\end{theorem}

The fact that $\delta_0$ is cyclic for $J$ is readily seen from the identity 
$p_{j}(J)\delta_{0}=\delta_{j}$, $j\geq0$, which follows from~\eqref{a.ogpol} by induction. The second claim of Theorem~\ref{thm.0a} is nothing but a particular case of the spectral theorem, see e.g.~\cite[Section~VII.2, Lemma~1]{RSI}. In the theory of orthogonal polynomials, this claim is sometimes referred to as the spectral theorem for orthogonal polynomials or Favard's theorem, see~\cite[Thm.~2.5.2]{ism_09} or~\cite[Theorem~1.3.7]{sim_11}.
The identity $p_{j}(J)\delta_{0}=\delta_{j}$ also readily implies that the polynomials $p_j$ defined by \eqref{a.ogpol} with $p_0=1$, satisfy the orthogonality relation
\begin{equation}
\int_{-\infty}^\infty p_j(s)p_k(s)\dd\mu(s)=\delta_{j,k}, \quad j,k\geq0.
\label{a.orth}
\end{equation}

It turns out that the image of the mapping $J\mapsto\mu$ consists of all non-degenerate compactly supported probability measures $\mu$ on $\bbR$. Moreover, this mapping is also injective. Both properties can be rephrased as the \emph{inverse spectral problem} for $J$ in the following form:

\begin{theorem}[Inverse spectral problem, self-adjoint case]\label{thm.0b}
$\, $ 
\begin{enumerate}[\rm (i)]
\item
Uniqueness: A bounded self-adjoint Jacobi matrix $J$ (satisfying $a_j>0$ for all $j\geq0$) is uniquely determined by its spectral measure $\mu$. 
\item
Surjectivity: For any non-degenerate (i.e. not reducing to a sum of finitely many atoms) probability measure $\mu$ with a bounded support, there is a bounded self-adjoint Jacobi matrix $J$ such that $\mu$ is the spectral measure of $J$. 
\end{enumerate}
\end{theorem}

In other words, there is a one-to-one correspondence $J\leftrightarrow \mu$ between the sets of operators and measures specified in Theorem~\ref{thm.0b}. See e.g.~\cite[Theorems~3.5 and A.~6]{ges-sim_jam97} for the proof and~\cite{akh_21} for the comprehensive theory behind it.

Moreover, the map $J\mapsto\mu$ is easily seen to be a homeomorphism with respect to appropriate topologies. A sequence of bounded self-adjoint Jacobi matrices $J_{N}$ converges strongly to a bounded self-adjoint Jacobi matrix $J$ if and only if the spectral measures $\mu_{N}$ are supported in a compact set (independent of $N$) and converge weakly to the spectral measure $\mu$: 
$$
 \int_{-\infty}^{\infty}f(s)\dd\mu_{N}(s) \to \int_{-\infty}^{\infty}f(s)\dd\mu(s), \quad \forall f\in C(\bbR),
$$
as $N\to\infty$.

\begin{remark}\label{rmk.a3}
Condition $a_j>0$ for all $j\geq0$ in \eqref{a.sa} is merely a normalisation condition.  The whole theory can be developed for the class of Jacobi matrices $J$ with arbitrary $a_{j}\neq0$ and a fixed sequence of signs of $a_j$. In fact, the procedure of reconstruction of $J$ from $\mu$ produces the sequences $b_j$, $a_j^2$, and after this the signs of $a_j$ can be prescribed in an arbitrary way. For example, taking $m=1,2$ in \eqref{a.moments}, we find
$$
\int_{-\infty}^\infty s \dd\mu(s)=
b_0 \quad\mbox{ and }\quad
\int_{-\infty}^\infty s^2 \dd\mu(s)=b_0^2+a_0^2,
$$
which determines $b_0$ and $a_0^2$. 
\end{remark}

\section{Main results}\label{sec.main}

\subsection{Bounded non-self-adjoint Jacobi matrices}

The purpose of this paper is to prove suitable analogues of Theorems~\ref{thm.0a} and \ref{thm.0b} for a class of bounded Jacobi matrices with \emph{complex} Jacobi parameters. In the rest of the paper, the matrix $J$ is given by \eqref{a.jacobi}, where $a_j$ and $b_j$ are bounded and satisfy \eqref{a2}. 
Assumption $a_j>0$ is the simplest normalisation condition on the arguments of $a_j$ and can be replaced by requiring that $a_j\not=0$ have any prescribed sequence of arguments, see Remark~\ref{rmk.arg} below. 

Crucially for our construction, $J$ is symmetric with respect to transposition. 
In other words, with $C$ denoting the usual complex conjugation $C: x\mapsto\overline{x}$ on $\ell^2(\bbN_0)$, we have  $CJ=J^*C$.
Such matrices are called \emph{$C$-symmetric}, see e.g. \cite{gar-put_tams06, gar-put_tams07}; we will say more on the $C$-symmetry structure below.

\subsection{The spectral measure $\nu$ and the direct spectral problem}
Let us consider the positive semi-definite operators $J^*J$ and $JJ^*$. The operators $J$ are in general not normal, i.e. $J^*J\not=JJ^*$. 
(It is easy to see that $J$ is normal if and only if $\Im b_j$ is a constant independent of $j$.)
As is standard, we write 
$$
\abs{J}=\sqrt{J^*J}\quad\text{ and }\quad \abs{J^*}=\sqrt{JJ^*};
$$ 
the square roots are taken in the sense of the usual functional calculus for self-adjoint operators. 
Consider the measure
\begin{equation}
\nu(\Delta)=\jap{\chi_\Delta(\abs{J})\delta_0,\delta_0},
\label{eq:def_nu}
\end{equation}
i.e. the spectral measure of the bounded self-adjoint operator $\abs{J}$ corresponding to the element $\delta_0$. 
Using the fact that complex conjugation $C$ effects an anti-unitary equivalence between $\abs{J}$ and $\abs{J^*}$ and leaves $\delta_0$ invariant, we see that $\nu$ can be alternatively written as
$$
\nu(\Delta)=\jap{\chi_\Delta(\abs{J^*})\delta_0,\delta_0}.
$$
Note that the measure $\nu$ is normalised by $\nu(\bbR)=1$. The measure $\nu$ is uniquely defined by its even moments
\begin{equation}
\jap{(J^*J)^n\delta_0,\delta_0}=\int_0^\infty s^{2n}\dd\nu(s),\quad n\geq0.
\label{a13a}
\end{equation}

Our first (simple) preliminary result is an analogue of Theorem~\ref{thm.0a}, i.e. the direct spectral problem for $J$. In order to state it, we recall the notions of multiplicity of spectrum and maximality. A bounded self-adjoint operator $T$ is said to have the \emph{spectrum of multiplicity} $\leq m$, if $T$ is unitarily equivalent to an orthogonal sum of no more than $m$ operators with simple spectrum. 
A vector $w$ is said to be an \emph{element of maximal type} with respect to a bounded self-adjoint operator $T$, if for any Borel set $\Delta\subset\bbR$ we have the implication 
$$
\chi_\Delta(T)w=0\quad\Rightarrow\quad \chi_\Delta(T)=0.
$$
If $T$ has a pure point spectrum, then the maximality of $w$ means that the projection of $w$ onto any eigenspace of $T$ is non-zero. 

\begin{theorem}[Direct spectral problem, non-self-adjoint case]\label{thm.a1}
For any bounded Jacobi matrix $J$ satisfying \eqref{a2}, the spectrum of $\abs{J}$ has multiplicity $\leq2$ and $\delta_0$ is an element of maximal type with respect to $\abs{J}$. In particular, if the spectrum of $\abs{J}$ is simple, then $\delta_0$ is a cyclic element for $\abs{J}$. 
\end{theorem}

The proof is given in Section~\ref{sec.b}.

\subsection{The phase function: definition} 

\begin{theorem}[Definition of $\psi$]\label{thm.a2}
For any bounded Jacobi matrix $J$ satisfying \eqref{a2}, there exists a unique function $\psi\in L^\infty(\nu)$ with $\psi(0)=1$ such that $\abs{\psi(s)}\leq 1$ for $\nu$-a.e. $s\geq0$ and 
\begin{equation}
\jap{Jf(\abs{J})\delta_0,\delta_0}
=
\int_0^\infty sf(s)\psi(s)\dd\nu(s), 
\quad
\forall f\in C(\bbR). 
\label{a10}
\end{equation}
\end{theorem}

We note that $\psi(0)=1$ is just a normalisation condition which ensures the uniqueness of $\psi$. 
Furthermore, if $\nu(\{0\})=0$, i.e. if $\Ker J=\{0\}$, this condition is not needed. The phase function $\psi$ is uniquely defined by the moments (compare with \eqref{a13a})
\begin{equation}
\jap{J(J^*J)^n\delta_0,\delta_0}=\int_0^\infty s^{2n+1}\psi(s) \dd\nu(s),\quad n\geq0,
\label{a13b}
\end{equation}
and the normalization $\psi(0)=1$.

Theorem~\ref{thm.a2} uses no specific features of Jacobi matrices. In fact, it holds true for any bounded $C$-symmetric operator $J$ and for any vector $\delta_0$ with $C\delta_0=\delta_0$. 
We provide two proofs of Theorem~\ref{thm.a2} in Section~\ref{sec.b}. The first proof is self-contained and uses no theory of $C$-symmetric operators. The second proof is based on the refined polar decomposition for $C$-symmetric operators, see e.g.~\cite[Theorem 2]{gar-put_tams07}; it emphasizes the link with the theory of $C$-symmetric operators and may be useful for possible generalizations.

\subsection{The main result}
The \emph{spectral data} of $J$ is the pair 
$$
\Lambda(J):=(\nu,\psi),
$$
where $\nu$ is defined by~\eqref{eq:def_nu} and $\psi$ is the phase function of Theorem~\ref{thm.a2}.
Our main result is
\begin{theorem}[Inverse spectral problem, non-self-adjoint case]\label{thm.a3}
$\, $ 
\begin{enumerate}[\rm (i)]
\item
Uniqueness: Any bounded Jacobi matrix $J$ satisfying \eqref{a2} is uniquely defined by its spectral data $\Lambda(J)$. 
\item
Surjectivity:
Let $\nu$ be a probability measure with a bounded infinite support in $[0,\infty)$. Let $\psi\in L^\infty(\nu)$ be a function such that $\abs{\psi(s)}\leq1$ for $\nu$-a.e. $s\geq0$ and $\psi(0)=1$. Then $(\nu,\psi)=\Lambda(J)$ for a bounded Jacobi matrix $J$ satisfying  \eqref{a2}.
\end{enumerate}
\end{theorem}
The proof of part (i) is given in Section~\ref{sec.c}.
The proof of part (ii) is given in Section~\ref{sec.d}.

It follows from the theorem that the spectral mapping $\Lambda: J\to(\nu,\psi)$ is a~bijection between the sets described in the statement of the theorem. In fact, this bijection is a homeomorphism with respect to suitable topologies. We will say that a sequence of spectral data $(\nu_N,\psi_N)$  \emph{converges weakly} to the spectral data $(\nu,\psi)$ (all $(\nu_N,\psi_N)$ and $(\nu,\psi)$ satisfy the constraints of Theorem~\ref{thm.a3}(ii)) if the measures $\nu_N$ are supported in a compact set (independent of $N$) and for any continuous function $f$ on $\bbR$ we have the two relations
\begin{align*}
\int_0^\infty f(s)\dd\nu_N(s)&\to\int_0^\infty f(s)\dd\nu(s), 
\\
\int_0^\infty f(s)\psi_N(s)d\nu_N(s)&\to\int_0^\infty f(s)\psi(s)d\nu(s)
\end{align*}
as $N\to\infty$. 

\begin{theorem}[Homeomorphism]\label{thm.a-homeo}
The bijection $J\mapsto \Lambda(J)$ is a homeomorphism with respect to the strong operator topology on the set of operators $J$ and the topology of weak convergence on the set of  spectral data. In other words, $J_N\to J$ strongly if and only if $\Lambda(J_N)\to\Lambda(J)$ weakly. 
\end{theorem}
The proof is given in Section~\ref{sec.e}.

\begin{remark}\label{rmk.arg}
It is easy to adapt Theorem~\ref{thm.a3} to the case when the off-diagonal Jacobi parameters $a_{j}$ are not assumed to be positive but instead their complex arguments are prescribed. In this case, the spectral mapping $\Lambda$ is a bijection between the set of bounded Jacobi operators with a prescribed sequence $\{\arg a_{j}\}_{j=0}^{\infty}$ and the set of pairs $(\nu,\psi)$ as described in Theorem~\ref{thm.a3}(ii).
In Section~\ref{sec.d7} we will comment on the very few places in the proofs that would change in this setting.
\end{remark}

\subsection{Specific classes of $J$}

Here we characterise specific classes of Jacobi matrices $J$ in terms of their phase function.

\begin{theorem}\label{thm.a3a}
$\, $ 
\begin{enumerate}[\rm (i)]
\item
The spectrum of $J^*J$ is simple if and only if $\abs{\psi(s)}=1$ for $\nu$-a.e. $s\geq0$.
\item
$J$ is self-adjoint if and only if $\psi(s)$ is real-valued for $\nu$-a.e. $s\geq0$. 
\item
We have $b_n=0$ for all $n$ if and only if $\psi(s)=0$ for $\nu$-a.e. $s>0$. 
\item
$J$ is self-adjoint and the spectrum of $J^2$ is simple if and only if $\psi(s)=\pm1$ for $\nu$-a.e. $s>0$. 
\item
If $J$ is normal (i.e. $J^*J=JJ^*$) and the spectrum of $\abs{J}$ is simple, then 
$$
J=\abs{J}\psi(\abs{J}).
$$
\end{enumerate}
\end{theorem}
The proof is given in Section~\ref{sec.e}.

\begin{remark*}
\begin{enumerate}[1.]
\item
We mention without proof that claim~(i) of Theorem~\ref{thm.a3a} can be strengthened as follows: the spectrum of $\abs{J}$ is simple on a Borel set $\Delta\subset[0,\infty)$ (i.e. the spectrum of the restriction of $\abs{J}$ onto $\Ran \chi_\Delta(\abs{J})$ is simple) if and only if $\abs{\psi(s)}=1$ for $\nu$-a.e. $s\in\Delta$.
\item 
Formula $J=\abs{J}\psi(\abs{J})$ from part~(v) of Theorem~\ref{thm.a3a} is essentially the polar decomposition of $J$. This gives some intuition into the phase function $\psi$. 
\end{enumerate}
\end{remark*}

If $J$ is self-adjoint, we can define both the spectral measure $\mu$ as in \eqref{eq.spmeasure} and the spectral data $(\nu,\psi)$. Let us discuss the connection between these parameters. For a measure $\mu$, we set
\begin{equation}
\widetilde\mu(\Delta):=\mu(-\Delta), \quad \Delta\subset\bbR, \quad -\Delta=\{-x: x\in\Delta\}.
\label{a.sharp}
\end{equation}
\begin{theorem}\label{thm.a3b}
For $J=J^*$ we have
\begin{align*}
\dd\mu(s)&=\frac{1+\psi(s)}{2}\dd\nu(s), \quad s\geq0,
\\
\dd\widetilde\mu(s)&=\frac{1-\psi(s)}{2}\dd\nu(s), \quad s>0.
\end{align*}
\end{theorem}
The proof is given in Section~\ref{sec.e}.

\subsection{The phase function: geometric interpetation} 
Here we give a geometric interpretation of the phase function by considering the simple case 
when $s>0$ is a singular value of $J$, i.e. when the subspace 
$$
E(s):=\Ker(J^*J-s^2I)
$$
is non-trivial. By Theorem~\ref{thm.a1}, the dimension of  $E(s)$ equals $1$ or $2$.  
Along with $E(s)$, we also consider $\overline{E(s)}=\Ker(JJ^*-s^2I)$. Note that $x\in E(s)$ if and only if $\overline{x}\in \overline{E(s)}$ and in particular the dimensions of $E(s)$ and $\overline{E(s)}$ coincide. 

Let $x\in E(s)$; applying $J$ to the equation $J^*Jx=s^2x$, we find $JJ^*(Jx)=s^2(Jx)$, i.e. $Jx\in\overline{E(s)}$. Thus, $J$ acts from $E(s)$ to $\overline{E(s)}$.
Moreover, for each $x\in E(s)$ we have 
$$
\norm{Jx}^2=\jap{J^*Jx,x}=s^2\norm{x}^2,
$$
i.e. $\frac1s J$ acts as an isometry from $E(s)$ onto $\overline{E(s)}$.
The phase function $\psi(s)$ is a parameter of this isometry, chosen as follows.

Let us first consider the case $\dim E(s)=\dim\overline{E(s)}=1$; then in order to describe the above isometry, we only need to specify the ``angle of rotation''. This angle is specified by reference to a pair of distinguished vectors in $E(s)$ and $\overline{E(s)}$. 
Let $h_s$ be the projection of $\delta_0$ onto $E(s)$ and $\overline{h_s}$ be the projection of $\delta_0$ onto $\overline{E(s)}$. We have
$$
\nu(\{s\})=\norm{h_s}^2=\norm{\overline{h_s}}^2>0
$$
by Theorem~\ref{thm.a1}. In this case $\psi(s)$ can be identified with the unimodular complex number in the identity
$$
\tfrac1s Jh_s=\psi(s)\overline{h_s}.
$$
Next, consider the case $\dim E(s)=2$. Let us consider the matrix of $\tfrac1sJ$ with respect to the basis $h_s,h_s^\perp$ in $E(s)$ and $\overline{h_s},\overline{h_s}^\perp$ in $\overline{E(s)}$. 
Then $\psi(s)$ can be identified with the top left entry of this matrix. In other words, 
$$
\tfrac1s\jap{Jh_s,\overline{h_s}}=\psi(s). 
$$
Since $\tfrac1s J$ is an isometry, we have $\abs{\psi(s)}\leq1$.

\subsection{Analogues of orthogonal polynomials for non-self-adjoint $J$}
Recall the recursive definition of classical orthogonal polynomials~\eqref{a.ogpol}. When the Jacobi parameters are allowed to be complex it is natural to study solutions of the \emph{antilinear} eigenvalue problem
\begin{equation}
 Jq(s)=s\overline{q}(s)
 \label{eq.anti.evlp}
\end{equation}
for $s\geq0$ (note the complex conjugation in the r.h.s.) and the infinite column vectors 
$$
q(s)=(q_{0}(s),q_{1}(s),\dots)\quad\text{ and }\quad\overline{q}(s)=(\overline{q}_{0}(s),\overline{q}_{1}(s),\dots),
$$ 
where $\overline{q}_{j}(s)\equiv\overline{q_j(s)}$. 
We fix the initial condition $q_{0}=1$ and introduce the sequence of polynomials $q_{j}$ defined uniquely (since $a_{j}\neq0$) by the recurrence 
\begin{equation}
\begin{aligned}
 b_{0}q_{0}(s)+a_{0}q_{1}(s)&=s\overline{q}_{0}(s),\\
 a_{j-1}q_{j-1}(s)+b_{j}q_{j}(s)+a_{j}q_{j+1}(s)&=s\overline{q}_{j}(s), \quad j\geq1.
\end{aligned}\label{eq.q.recur}
\end{equation}
Clearly, if the Jacobi parameters $a_{j}$ and $b_{j}$ are real and $a_j>0$, the polynomials $q_{j}$ coincide with the usual orthogonal polynomials $p_{j}$. 

A consideration of the antilinear eigenvalue problem~\eqref{eq.anti.evlp} for complex symmetric matrices goes back at least to Takagi~\cite{tak_24} and has been studied in greater generality for $C$-symmetric unbounded operators with compact resolvent in~\cite{gar-put_tams07}; it also plays an important role in the inverse spectral theory for compact Hankel operators (see e.g. \cite{AAK,GG1}).
Our goal here is to point out an orthogonality relation for polynomials $q_{j}$ in terms of the spectral data $(\nu,\psi)$ corresponding to the Jacobi matrix $J$ with the Jacobi parameters $a_{j}$ and $b_{j}$.

\begin{theorem}\label{thm.og}
Let $J$ be a bounded Jacobi matrix satisfying~\eqref{a2}, and let $q_j(s)$ be the polynomials defined by \eqref{eq.q.recur} with $q_0=1$. 
Then, for all $j,k\geq0$, we have the orthogonality relation
$$
\frac12\int_0^\infty
\Jap{\begin{pmatrix}
1+\Re \psi(s)& -\ii\Im\psi(s)
\\
\ii\Im\psi(s)&1-\Re\psi(s)
\end{pmatrix}
\begin{pmatrix}q_j(s)\\q_j(-s)\end{pmatrix},
\begin{pmatrix}q_k(s)\\q_k(-s)\end{pmatrix}
}_{\bbC^2}\dd\nu(s)=\delta_{j,k},
$$
or alternatively,
$$
\int_0^\infty
\Jap{\begin{pmatrix}
1& \overline{\psi(s)}
\\
\psi(s)&1
\end{pmatrix}
\begin{pmatrix}q_{j}^{\rm{e}}(s) \\ q_{j}^{\rm{o}}(s)\end{pmatrix},
\begin{pmatrix}q_{k}^{\rm{e}}(s) \\ q_{k}^{\rm{o}}(s)\end{pmatrix}
}_{\bbC^2}\dd\nu(s)=\delta_{j,k},
$$
where
$$
 q_{j}^{\rm{e}}(s):=\frac{q_{j}(s)+q_{j}(-s)}{2} \quad\mbox{ and }\quad q_{j}^{\rm{o}}(s):=\frac{q_{j}(s)-q_{j}(-s)}{2}.
$$

\end{theorem}

The proof is given at the end of Section~\ref{sec.e}.

\begin{remark*}
 If $J=J^{*}$, then we see from Theorem~\ref{thm.a3b} that the orthogonality relation for $q_{j}\equiv p_{j}$ from Theorem~\ref{thm.og} simplifies to the standard form \eqref{a.orth}.
\end{remark*}

\subsection{Relation to self-adjoint block Jacobi matrices}
The key idea of our approach is in passing from non-self-adjoint $J$ to a \emph{self-adjoint} block $2\times 2$ Jacobi matrix 
$$
\mathbb{J}=\begin{pmatrix}
B_0 & A_0 & 0 & 0 & 0 &\cdots\\
A_0^* & B_1 & A_1 & 0 & 0 & \cdots\\
0 & A_1^* & B_2 & A_2 & 0 &\cdots\\
0 & 0 & A_2^* & B_3 & A_3 & \cdots\\
\vdots&\vdots&\vdots&\vdots&\vdots&\ddots
\end{pmatrix},
\quad
  A_{j}=\begin{pmatrix}
   0 & a_{j} \\ a_{j} & 0
  \end{pmatrix}
  \quad\mbox{ and }\quad   
  B_{j}=\begin{pmatrix}
   0 & b_{j} \\ \bar{b}_{j} & 0
  \end{pmatrix}.
$$
We relate the spectral data $(\nu,\psi)$ of $J$ to the $2\times 2$ spectral measure of $\mathbb J$. After that, we use available results on the inverse spectral problem for block Jacobi matrices $\mathbb J$. We review these results in Section~\ref{sec.cc}.

\subsection{Related literature}

In~\cite{gus_mz78}, Guseinov considers complex Jacobi matrices $J$ (without the condition $a_j>0$) and defines the generalised spectral function $\mathcal P$ as the linear functional on polynomials defined by (in our notation)
$$
\mathcal P[f]=\jap{f(J)\delta_0,\delta_0}.
$$
He then establishes a bijection between all complex Jacobi matrices and all linear functionals $\mathcal P$ of a certain class. Although $\mathcal P[f]$ is reminiscent of our $\jap{f(\abs{J})\delta_0,\delta_0}$, it is difficult to relate these quantities; moreover, no effective parameterisation of the admissible set of generalised spectral functions $\mathcal P$ is suggested in \cite{gus_mz78}.

In \cite{HuhP}, Huhtanen and Per\"{a}m\"{a}ki consider a  class of measures $\rho$ in the complex plane that they call biradial. With every biradial measure they associate a bounded Jacobi matrix $J$ (with $a_j>0$) such that (in our notation)
\begin{align*}
\jap{(J^*J)^n\delta_0,\delta_0}&=\int_\bbC \abs{z}^{2n}\dd\rho(z),
\\
\jap{J(J^*J)^n\delta_0,\delta_0}&=\int_\bbC z\abs{z}^{2n}\psi(z)\dd\rho(z),
\end{align*}
for every $n\geq0$. Clearly, this is related to \eqref{a13a}, \eqref{a13b}; however, because the class of biradial measures  $\rho$ is ``too wide'', the map from $\rho$ to $J$ turns out to be surjective but not injective.

The direct and inverse spectral theory for finite and semi-infinite non-self-adjoint Jacobi matrices with rank-one imaginary part is developed in~\cite{arl-tse_jfa06}. The language used in~\cite{arl-tse_jfa06} is quite different from ours and in particular the focus is on the spectrum of $J$ rather than $\abs{J}$.

In the survey~\cite{Beck}, non-self-adjoint symmetric Jacobi matrices are reviewed from the point of view of their connection to \emph{formal orthogonal polynomials}, i.e. the polynomials defined recursively by~\eqref{a.ogpol}. 

A lot of work has been done recently towards establishing Lieb-Thirring inequalities for complex Jacobi matrices, see e.g. \cite{Gol} and references therein. In this line of research, one deals with eigenvalue estimates for tri-diagonal non-self-adjoint and not necessarily symmetric matrices that are compact perturbations of the standard Jacobi matrix $J_0$ corresponding to $a_j=1$ and $b_j=0$ for all $j$.

As already mentioned, Jacobi matrices satisfying \eqref{a.sa} belong to the class of $C$-symmetric operators. There has been much work on the theory of these operators, see e.g. \cite{gar-put_tams06, gar-put_tams07} or the survey~\cite{gar-pro-put_jpa14}. We do not explicitly use any of the theory of $C$-symmetric operators, even though the $C$-symmetric structure is crucial for our construction. 

There are some similarities between the construction of this paper and the inverse spectral theory for compact non-self-adjoint Hankel matrices, developed in \cite{GG1,GG2}. The common feature is that Hankel matrices are $C$-symmetric. However, it turns out that in order to have uniqueness, the spectral data needs to contain more information, and moreover the surjectivity question for non-compact Hankel matrices is much more complicated, see \cite{GPT}.

\section{Direct problem: proofs of Theorems~\ref{thm.a1} and \ref{thm.a2}}\label{sec.b}

\subsection{The distinguished element of maximal type}

Let us first prove Theorem~\ref{thm.a1}. 
Before embarking on the proof, we need a definition and two lemmas.
If $x,y\in\ell^2(\bbN_0)$ and $T$ is a bounded operator in $\ell^2(\bbN_0)$, we will denote by $\Span_T(x,y)$ the minimal closed invariant subspace of $T$ containing both $x$ and $y$. In other words, $\Span_T(x,y)$ is the closure of the set of all finite linear combinations of elements $T^m x$ and $T^n y$ over $m,n\geq0$. 

\begin{lemma}\label{lma.b1}
We have $\Span_{\abs{J}}(\delta_0,J^*\delta_0)=\ell^2(\bbN_0)$. 
In particular, the multiplicity of the spectrum of $\abs{J}$ is $\leq2$. 
\end{lemma}
\begin{proof}
Since $\abs{J}=\sqrt{J^*J}$, it suffices to prove that $\Span_{J^*J}(\delta_0,J^*\delta_0)=\ell^2(\bbN_0)$. Further, since $J^*\delta_0=\overline{b_0}\delta_0+a_0\delta_1$, it suffices to prove that $\Span_{J^*J}(\delta_0,\delta_1)=\ell^2(\bbN_0)$. Let us prove that all elements of the standard basis belong to $\Span_{J^*J}(\delta_0,\delta_1)$; clearly, the required statement will follow from here. 

It is straightforward to see that $J^*J$ is a five-diagonal matrix, with non-zero elements in the second sub-diagonal (and second super-diagonal); in particular, 
$$
[J^*J]_{j+2,j}=a_ja_{j+1}\not=0.
$$
For example, 
$$
J^*J\delta_0=a_0 a_1\delta_2+\text{(linear combination of $\delta_0$ and $\delta_1$)}.
$$
It follows that $\delta_2\in\Span_{J^*J}(\delta_0,\delta_1)$. Next, 
$$
J^*J\delta_1=a_1a_2\delta_3+\text{(linear combination of $\delta_0$, $\delta_1$, $\delta_2$)},
$$
and therefore $\delta_3\in\Span_{J^*J}(\delta_0,\delta_1)$. Continuing in the same way, we find that $\delta_j\in\Span_{J^*J}(\delta_0,\delta_1)$ for all $j\geq1$. 
The proof is complete. 
\end{proof}

\begin{lemma}\label{lma.b2}
For any bounded Borel function $f$, we have
$$
Jf(\abs{J})=f(\abs{J^*})J\quad\text{ and }\quad
J^*f(\abs{J^*})=f(\abs{J})J^*. 
$$
\end{lemma}
\begin{proof}
For all $n\geq0$, we have $J(J^*J)^n=(JJ^*)^nJ$. Taking linear combinations, we obtain the relation $Jg(J^*J)=g(JJ^*)J$ for any polynomial $g$. By a standard approximation argument, we deduce the same relation for any bounded Borel function $g$. This can be rewritten as $Jf(\abs{J})=f(\abs{J^*})J$ with $f(s)=g(s^2)$. The second relation is obtained by interchanging $J$ and $J^*$.
\end{proof}

\begin{proof}[Proof of Theorem~\ref{thm.a1}]
The first claim of the theorem (the spectrum of $\abs{J}$ has multiplicity $\leq2$) has already been established in Lemma~\ref{lma.b1}. 

As is well known  \cite[Theorem~7.3.4]{BS}, any self-adjoint operator has elements of maximal type. Let $w$ be an element of maximal type for $\abs{J}$. Let us prove that $\delta_0$ is also of maximal type. Suppose that for some Borel $\Delta\subset[0,\infty)$ we have $\chi_\Delta(\abs{J})\delta_0=0$ or equivalently $\delta_0\perp\Ran \chi_\Delta(\abs{J})$. Denote $v=\chi_\Delta(\abs{J})w$. 
For all $m\geq0$, we have $\delta_0\perp \abs{J}^m v$, or equivalently 
\begin{equation}
v\perp\abs{J}^m\delta_0, \quad m\geq0.
\label{b3}
\end{equation}
Further, for any $m\geq0$, we have
$$
\jap{v, \abs{J}^mJ^*\delta_0}
=
\jap{w,\chi_\Delta(\abs{J})\abs{J}^mJ^*\delta_0}
=
\jap{w,\abs{J}^m\chi_\Delta(\abs{J})J^*\delta_0}.
$$
Using Lemma~\ref{lma.b2} for $f=\chi_\Delta$, we continue
$$
\jap{w,\abs{J}^m\chi_\Delta(\abs{J})J^*\delta_0}
=
\jap{w,\abs{J}^mJ^*\chi_\Delta(\abs{J^*})\delta_0}
=
\jap{w,\abs{J}^mJ^*\overline{\chi_\Delta(\abs{J})\delta_0}}
=0
$$
for any $m\geq0$. Thus, $v\perp \abs{J}^mJ^*\delta_0$ for all $m\geq0$. Putting this together with \eqref{b3}, we find $v\perp\Span_{\abs{J}}(\delta_0,J^*\delta_0)$, and so by Lemma~\ref{lma.b1} we conclude that  $v=0$. Recalling the definition of $v$, we find $\chi_\Delta(\abs{J})w=0$; by the maximality of $w$ this yields $\chi_\Delta(\abs{J})=0$. The proof of the maximality of $\delta_0$ is complete. 

Finally, the last claim of the theorem (if the spectrum of $\abs{J}$ is simple, then $\delta_0$ is cyclic) follows from the general statement \cite[Theorem~7.3.2(c)]{BS}: \emph{If $T$ is a self-adjoint operator with simple spectrum, then any element of  maximal type is cyclic for $T$. } 
\end{proof}

\subsection{The first proof of Theorem~\ref{thm.a2}}
\emph{Step 1:}
Let $f$ be a continuous function vanishing in a neighbourhood of the origin. Let us prove the bound
\begin{equation}
\abs{\jap{Jf(\abs{J})\delta_0,\delta_0}}
\leq
\int_0^\infty s\abs{f(s)}\dd\nu(s). 
\label{b7}
\end{equation}
We write $f(s)=\abs{f(s)}^{1/2}f(s)^{1/2}$, where
$$
f(s)^{1/2}:=
\begin{cases}
f(s)/\abs{f(s)}^{1/2}, & \text{ if $f(s)\not=0$,}
\\
0, & \text{if $f(s)=0$.}
\end{cases}
$$
With this notation, for $s>0$ we set  
$$
g(s)=s^{1/2}\abs{f(s)}^{1/2},
\qquad
h(s)=s^{-1/2}f(s)^{1/2},
$$
and $g(0)=h(0)=0$. With this definition, both $h$ and $g$ are bounded and continuous on $[0,\infty)$ and $f(s)=g(s)h(s)$. From here, using Lemma~\ref{lma.b2} at the second step, we find
\begin{align}
\jap{Jf(\abs{J})\delta_0,\delta_0}
&=
\jap{Jg(\abs{J})h(\abs{J})\delta_0,\delta_0}
=
\jap{g(\abs{J^*})Jh(\abs{J})\delta_0,\delta_0}
\notag
\\
&=
\jap{Jh(\abs{J})\delta_0,g(\abs{J^*})\delta_0}.
\label{b7a}
\end{align}
Furthermore, 
\begin{align*}
\norm{Jh(\abs{J})\delta_0}^2
&=
\jap{J^*Jh(\abs{J})\delta_0,h(\abs{J})\delta_0}
=
\int_0^\infty s^2\abs{h(s)}^2\dd\nu(s)
=
\int_0^\infty s\abs{f(s)}\dd\nu(s),
\\
\norm{g(\abs{J^*})\delta_0}^2
&=
\int_0^\infty \abs{g(s)}^2\dd\nu(s)
=
\int_0^\infty s\abs{f(s)}\dd\nu(s).
\end{align*}
Putting this together and using Cauchy-Schwarz for the right hand side of \eqref{b7a}, we obtain \eqref{b7}. 

\emph{Step 2:} Clearly, $\jap{Jf(\abs{J})\delta_0,\delta_0}$ is a linear functional in $f$, and the previous step shows that it is bounded on a dense set in $L^1(s\dd\nu(s))$ with norm $\leq1$. From here we obtain the desired representation \eqref{a10} with a unique $\psi\in L^\infty(sd\nu(s))$ satisfying $\abs{\psi(s)}\leq1$ for $\nu$-a.e. $s>0$. The value $\psi(0)$ does not affect the r.h.s. of \eqref{a10} and can be set arbitrarily. We set $\psi(0)=1$ to satisfy the requirement of the theorem. This ensures the uniqueness of $\psi$ as an element of $L^\infty(\nu)$. 
The proof of Theorem~\ref{thm.a2} is complete.
\qed

\subsection{The second proof of Theorem~\ref{thm.a2}}

Here we prove a slightly more general statement than Theorem~\ref{thm.a2}. In order to state it, we need to introduce the following notation. Let $\calH_0$ (resp. $\overline{\calH_0}$) be the minimal closed invariant subspace of $|J|$ (resp $|J^*|$) containing $\delta_{0}$, and let $\calP_0$ (resp. $\overline{\calP_0}$) be the orthogonal projection onto $\calH_0$ (resp. $\overline{\calH_0}$) in $\ell^2(\bbN_0)$.

\begin{proposition}\label{prop.psi.strong.def}
There exists a unique function $\psi\in L^\infty(\nu)$ with $\psi(0)=1$ such that $\abs{\psi(s)}\leq 1$ for $\nu$-a.e. $s\geq0$ and for all $f\in C(\bbR)$, 
\begin{align}
\overline{\calP_0}Jf(\abs{J})\delta_0&=\abs{J^*}\psi(\abs{J^*})f(\abs{J^*})\delta_0,
\label{eq.strong.psi}
\\
\calP_0 J^*f(\abs{J^*})\delta_0&=\abs{J}\overline{\psi}(\abs{J})f(\abs{J})\delta_0.
\label{eq.strong.psi.adjoint}
\end{align}
\end{proposition}

Taking the inner product of both sides in~\eqref{eq.strong.psi} with $\delta_{0}$, we readily arrive at Theorem~\ref{thm.a2}.

\begin{proof}[Proof of Proposition~\ref{prop.psi.strong.def}]
First note that identities \eqref{eq.strong.psi} and \eqref{eq.strong.psi.adjoint} are complex conjugated to each other. Therefore it suffices to prove~\eqref{eq.strong.psi}.

\emph{Step 1: Existence of $\psi\in L^{2}(\nu)$ satisfying~\eqref{eq.strong.psi}.} The refined polar decomposition of the $C$-symmetric Jacobi operator~$J$ yields the formula
\begin{equation}
 J=C\mathcal{I}|J|,
\label{eq.refin.polar.decomp}
\end{equation}
where $\mathcal{I}$ is an antilinear involutive partial isometry which commutes with $|J|$; see~\cite[Theorem 2]{gar-put_tams07} for details.

Pick an arbitrary $f\in C(\bbR)$ and apply both sides of~\eqref{eq.refin.polar.decomp} to the vector $f(|J|)\delta_{0}$:
\begin{equation}
Jf(|J|)\delta_{0}=C\mathcal{I}|J|f(|J|)\delta_{0}.
\label{eq.id.f.delta_0}
\end{equation}
It follows from the identity $C|J|=|J^{*}|C$ that 
$$
 Cg(|J|)=\overline{g}(|J^{*}|)C
$$
for any $g\in C(\bbR)$. Similarly, from the commutation relation $\mathcal{I}|J|=|J|\mathcal{I}$, one infers that 
$$
\mathcal{I}g(|J|)=\overline{g}(|J|)\mathcal{I}
$$
(the complex conjugation of $g$ is due to the antilinearity of $\mathcal{I}$). Hence we may rewrite~\eqref{eq.id.f.delta_0} as
$$
Jf(|J|)\delta_{0}=|J^{*}|f(|J^{*}|)C\mathcal{I}\delta_{0}.
$$

Next we apply the projection $\overline{\calP_0}$ to both sides of the last identity. Since $\overline{\calH_0}$ reduces $|J^{*}|$, i.e. $\overline{\calP_0}|J^*|=|J^*|\overline{\calP_0}$, we have the commutation relation 
$$
\overline{\calP_0}\,g(|J^*|)=g(|J^{*}|)\overline{\calP_0}
$$
for any $g\in C(\bbR)$. From here we find the identity 
\begin{equation}
\overline{\calP_0}Jf(|J|)\delta_{0}=|J^{*}|f(|J^{*}|)\overline{\calP_0}C\mathcal{I}\delta_{0}.
\label{eq.id.f.delta_0_projected}
\end{equation}

Since $\overline{\calP_0}C\mathcal{I}\delta_{0}\in\overline{\calH_{0}}$ and $\delta_{0}$ is a cyclic vector for operator $|J^{*}|$ restricted to $\overline{\calH_{0}}$, there exists $\psi\in L^{2}(\nu)$ such that 
$$
\overline{\calP_0}C\mathcal{I}\delta_{0}=\psi(|J^{*}|)\delta_{0}.
$$
Plugging this into~\eqref{eq.id.f.delta_0_projected}, we obtain~\eqref{eq.strong.psi}.

\emph{Step 2: The function $\psi$ satisfies $|\psi|\leq1$ $\nu$-a.e. and $\psi(0)=1$.}
This argument is similar to the first proof of Theorem~\ref{thm.a2}. 
Consider the norms on both sides of~\eqref{eq.strong.psi}. For the left hand side we have
$$
\norm{\overline{\calP_0}Jf(\abs{J})\delta_0}^2
\leq
\norm{Jf(\abs{J})\delta_0}^2
=
\jap{J^*Jf(\abs{J})\delta_0,f(\abs{J})\delta_0}
=
\int_0^\infty s^2\abs{f(s)}^2\dd\nu(s),
$$
and for the right hand side
$$
\norm{|J^*|\psi(\abs{J^*})f(\abs{J^*})\delta_0}^2=\int_0^\infty s^{2}\abs{\psi(s)}^2\abs{f(s)}^2\dd\nu(s).
$$
Combining this, we find
$$
\int_0^\infty s^{2}\abs{\psi(s)}^2\abs{f(s)}^2 \dd\nu(s)
\leq 
\int_0^\infty s^2\abs{f(s)}^2\dd\nu(s)
$$
for all $f\in C(\bbR)$ and therefore $\abs{\psi(s)}\leq1$ for $\nu$-a.e. $s>0$. 
The value $\psi(0)$ does not affect~\eqref{eq.strong.psi} and can be set arbitrarily; to comply with the claim of the theorem, we set $\psi(0)=1$. 

\emph{Step 3: The uniqueness of $\psi$.} 
Suppose \eqref{eq.strong.psi} holds true with $\psi=\psi_1$ and with $\psi=\psi_2$ where $\psi_1$ and $\psi_2$ are two elements of $L^2(\nu)$. Taking the inner product with $\delta_0$ and observing that the left hand side of \eqref{eq.strong.psi} is independent of $\psi$, we find
$$
\int_0^\infty s\psi_1(s)f(s)d\nu(s)=\int_0^\infty s\psi_2(s)f(s)d\nu(s)
$$
for all $f\in C(\bbR)$. It follows that $s\psi_1(s)=s\psi_2(s)$ for $\nu$-a.e. set of $s\geq0$. We conclude that $\psi_1(s)=\psi_2(s)$ for $\nu$-a.e. $s>0$. Finally, we have $\psi_1(0)=\psi_2(0)=1$ by our normalisation assumption. 
This proves the uniqueness of $\psi$. 
\end{proof}

\section{Block Jacobi matrices}\label{sec.cc}

\subsection{Self-adjoint block Jacobi matrices}
Here we recall the basics of the spectral theory of self-adjoint block Jacobi matrices
\begin{equation}
\mathbb{J}=\begin{pmatrix}
B_0 & A_0 & 0 & 0 & 0 &\cdots\\
A_0^* & B_1 & A_1 & 0 & 0 & \cdots\\
0 & A_1^* & B_2 & A_2 & 0 &\cdots\\
0 & 0 & A_2^* & B_3 & A_3 & \cdots\\
\vdots&\vdots&\vdots&\vdots&\vdots&\ddots
\end{pmatrix},
\label{c.defblockJ}
\end{equation}
where the Jacobi parameters $A_{j},B_{j}$ are $2\times 2$ matrices, and $A_j^*$ denotes the adjoint of $A_j$. In general, the Jacobi parameters can be $N\times N$ matrices for any fixed $N$, but here we are interested in the $2\times 2$ case. As a main source, we use the survey~\cite{dam-pus-sim_sat08}.

The matrices $B_{j}$ are assumed to be self-adjoint and $A_{j}$ non-singular, i.e., $\det A_{j}\neq0$. In addition, we also suppose that 
$$
\sup_j \|A_{j}\|<\infty\quad\text{ and }\quad\sup_j \|B_{j}\|<\infty,
$$
where $\|\cdot\|$ can be any norm on the space of $2\times2$ matrices. Then $\mathbb{J}$, regarded as an operator on the Hilbert space $\ell^{2}(\bbN_{0};\bbC^{2})$, is bounded and self-adjoint. Recall that the elements of $\ell^{2}(\bbN_{0};\bbC^{2})$ are sequences $X=(X_0,X_1,X_2,\dots)$, where every coordinate $X_{j}$ is a~vector in $\bbC^{2}$ and $\sum_{j=0}^{\infty}\|X_{j}\|_{\bbC^2}^{2}<\infty$, where $\|\cdot\|_{\bbC^2}$ is the usual Euclidean norm in $\bbC^2$. 

Two block Jacobi matrices $\mathbb{J}$ and $\mathbb{J}'$ are called \emph{equivalent}, $\mathbb{J}\sim\mathbb{J}'$, if 
$$
\mathbb{J}'=\bbD(U)^{*}\mathbb{J}\bbD(U),
$$
where $\bbD(U)$ is a block diagonal matrix
\begin{equation}
 \bbD(U)=\begin{pmatrix}
U_0 & 0 & 0 & 0 & \cdots\\
0 & U_1 & 0 & 0 & \cdots\\
0 & 0 & U_2 & 0 & \cdots\\
0 & 0 & 0 & U_{3} & \cdots\\
\vdots&\vdots&\vdots&\vdots&\ddots
\end{pmatrix},
\label{c.defu}
\end{equation}
and $U=\{U_{j}\}_{j=0}^{\infty}$ is a sequence of $2\times 2$ unitary matrices with $U_{0}=I$, the identity matrix. We will denote by $[\bbJ]$ the equivalence class containing $\bbJ$.

Further, we denote by $\mathbb{P}_{0}:\ell^{2}(\bbN_{0};\bbC^{2})\to \bbC^2$ the orthogonal projection onto the first coordinate, i.e., 
$$
\mathbb{P}_{0}: X=(X_0,X_1,X_2,\dots)\mapsto X_{0}.
$$
The \emph{spectral measure} $\meas$ of $\mathbb{J}$ is the $2\times 2$ matrix-valued measure uniquely determined by the relation
$$
\mathbb{P}_{0}f(\mathbb{J})\mathbb{P}_{0}^*
=
\int_{-\infty}^\infty f(s)\dd\meas(s)
$$
for all continuous functions $f$ on $\bbR$.

We wish to define the corresponding space $L^2(\meas)$. 
For a $\bbC^2$-valued polynomial~$P$ on the real line, set 
$$
\norm{P}_{L^2(\meas)}^2:=
\int_{-\infty}^\infty \jap{\dd\meas(s)P(s),P(s)}_{\bbC^2}. 
$$
\begin{theorem} 
\cite[Thm.~2.11]{dam-pus-sim_sat08}
\label{thm.cc1}
For any bounded non-singular (i.e. $\det A_j\not=0$) block Jacobi matrix $\bbJ$, the corresponding measure $\meas$ is \textbf{non-degenerate}, i.e. $\norm{P}_{L^2(\meas)}\not=0$ for any non-zero $\bbC^{2}$-valued polynomial $P$. Let $L^2(\meas)$ be the completion of the set of such polynomials with respect to the norm $\norm{\cdot}_{L^2(\meas)}$. Then $\bbJ$ is unitarily equivalent to the operator of multiplication by the independent variable in $L^2(\meas)$. 
\end{theorem}

For a scalar measure $\meas$, non-degeneracy in the above sense means that the measure does not reduce to the sum of finitely many point masses. For matrix measures, this condition is more subtle, as can be seen by looking at the direct sum of a degenerate and non-degenerate scalar measure. See \cite[Lemma~2.1]{dam-pus-sim_sat08} for a~more detailed discussion.

It is straightforward to see that two equivalent Jacobi matrices $\mathbb{J}$ and $\mathbb{J}'$ give rise to the same spectral measure; here it is important that $U_{0}=I$ in \eqref{c.defu}. Moreover, we have 
\begin{theorem}\cite[Thm.~2.11]{dam-pus-sim_sat08}
\label{thm.cc2}
\begin{enumerate}[\rm (i)]
\item
Uniqueness: 
The equivalence class $[\bbJ]$ is uniquely defined by the spectral measure $\meas$ of $\bbJ$. 
\item
Surjectivity: 
Let $\meas$ be a non-degenerate $2\times2$ matrix-valued measure with compact support on $\bbR$. Then there exists a non-singular (i.e. $\det A_j\not=0$) $2\times2$ block Jacobi matrix $\bbJ$ such that $\meas$ is the spectral measure of $\bbJ$. 
\end{enumerate}
\end{theorem}

\subsection{The matrices $J$, $\bJ$ and $\bbJ$}

For a bounded Jacobi matrix $J$ satisfying \eqref{a2}, let us define the block Jacobi matrix $\mathbb{J}$ as in \eqref{c.defblockJ} with the Jacobi parameters
\begin{equation}
  A_{j}=\begin{pmatrix}
   0 & a_{j} \\ a_{j} & 0
  \end{pmatrix}
  \quad\mbox{ and }\quad   
  B_{j}=\begin{pmatrix}
   0 & b_{j} \\ \bar{b}_{j} & 0
  \end{pmatrix}.
\label{cc1}
\end{equation}
We claim that $\bbJ$, as an operator in $\ell^{2}(\bbN_{0};\bbC^{2})$, is unitarily equivalent to the off-diagonal operator-valued matrix
$$
 \bJ:= \begin{pmatrix}
  0 & J \\ J^{*} & 0
  \end{pmatrix}
\quad \text{ on $\ell^{2}(\bbN_{0})\oplus\ell^{2}(\bbN_{0})$}.
$$
Verification of this claim is just a matter of rearranging the standard basis in $\ell^{2}(\bbN_{0};\bbC^{2})$. Indeed, let 
$$ 
E_{j}^{\uparrow}=\begin{pmatrix} 1\\0 \end{pmatrix}\delta_{j} 
\quad\mbox{ and }\quad 
E_{j}^{\downarrow}=\begin{pmatrix} 0\\1 \end{pmatrix}\delta_{j}, \quad j\geq0,
$$
be the vectors of the standard basis of $\ell^{2}(\bbN_{0};\bbC^{2})$.
Let
$$
V:\ell^{2}(\bbN_{0};\bbC^{2})\to\ell^{2}(\bbN_{0})\oplus\ell^{2}(\bbN_{0})
$$ 
be the isometric isomorphism defined by the formulas
$$
VE_{j}^{\uparrow}:=\delta_{j}\oplus0 \quad\mbox{ and }\quad VE_{j}^{\downarrow}:=0\oplus \delta_{j}, \quad j\geq0.
$$
Then one readily verifies that 
 \begin{equation}
  \mathbb{J}=V^{*}\bJ V.
 \label{eq.unit.equiv.compact}
 \end{equation}

Using this unitary equivalence, let us express the spectral measure of $\bbJ$ in terms of $\bJ$. Let $\bP_0:\ell^2(\bbN_0)\oplus\ell^2(\bbN_0)\to\bbC^2$ be the orthogonal projection onto the two-dimensional subspace spanned by $\delta_0\oplus 0$ and $0\oplus\delta_0$. Explicitly, if $u,v\in\ell^2(\bbN_0)$, then 
$$
\bP_0: u\oplus v\mapsto \begin{pmatrix}u_0\\ v_0\end{pmatrix}\in\bbC^2.
$$
Clearly, we have $\mathbb{P}_{0}=\bP_0V$ and $\mathbb{P}_{0}V^*=\bP_0$.
From here and \eqref{eq.unit.equiv.compact} we find 
\begin{equation}
\bP_0f(\bJ)\bP_0^*
=
\bbP_0V^*f(\bJ)V\bbP_0^*
=
\bbP_0f(\bbJ)\bbP_0^*
=
\int_{-\infty}^\infty f(s)\dd\meas(s). 
\label{c1}
\end{equation}

\subsection{The computation of $\mathbb{P}_{0}\mathbb{J}^{m}\mathbb{P}_{0}^*$}
For any $n\geq0$, a direct computation yields
$$
\bJ^{2n}=
\begin{pmatrix}
(JJ^{*})^{n}&0
\\
0&(J^*J)^{n}
\end{pmatrix} 
\quad\mbox{ and }\quad
\bJ^{2n+1}=
\begin{pmatrix}
0&J(J^*J)^n
\\
J^*(JJ^{*})^n&0
\end{pmatrix}.
$$
Taking into account the complex symmetry of $J$, we find for all $n\geq0$ that
\begin{equation}
\mathbb{P}_{0}\mathbb{J}^{2n}\mathbb{P}_{0}^*
=
\bP_0\bJ^{2n}\bP_0^*
=
\begin{pmatrix}
\omega_{2n}&0
\\
0&\omega_{2n}
\end{pmatrix},
\label{cc2}
\end{equation}
where
$$
\omega_{2n}:=\jap{(J^*J)^{n}\delta_0,\delta_0}=\jap{(JJ^{*})^{n}\delta_0,\delta_0},
$$
and similarly
\begin{equation}
\\
\mathbb{P}_{0}\mathbb{J}^{2n+1}\mathbb{P}_{0}^*
=
\bP_0\bJ^{2n+1}\bP_0^*
=
\begin{pmatrix}
0&\omega_{2n+1}
\\
\overline{\omega}_{2n+1}&0
\end{pmatrix},
\label{cc3}
\end{equation}
where 
$$
\omega_{2n+1}:=\jap{J(J^{*}J)^{n}\delta_0,\delta_0}=\overline{\jap{J^*(JJ^*)^{n}\delta_0,\delta_0}}.
$$
Below we will use these formulas.

\subsection{A combinatorial lemma for $\mathbb{P}_{0}\mathbb{J}^{m}\mathbb{P}_{0}^*$}
Let $\bbJ$ be as in \eqref{c.defblockJ} with $\det A_j\not=0$ for all $j$; here we do not assume the structure \eqref{cc1} of the Jacobi parameters $A_j$, $B_j$. 
In what follows, we will need induction arguments that involve the expressions
 $\mathbb{P}_{0}\mathbb{J}^{m}\mathbb{P}_{0}^*$. The dependence of these expressions on the Jacobi parameters $A_{j}$ and $B_{j}$ is combinatorially complicated. Here we prepare an auxiliary statement on this point. 

\begin{lemma}\label{lem.moments}
 For all $n\in\bbN$, one has
 $$
  \mathbb{P}_{0}\mathbb{J}^{2n}\mathbb{P}_{0}^*=A_{0}A_{1}\cdots A_{n-2}A_{n-1}A_{n-1}^{*}A_{n-2}^{*}\cdots A_{1}^{*}A_{0}^{*}+\mathcal{Y}_{2n}
 $$
 and
\begin{equation}
  \mathbb{P}_{0}\mathbb{J}^{2n+1}\mathbb{P}_{0}^*=A_{0}A_{1}\cdots A_{n-1}B_{n}A_{n-1}^{*}\cdots A_{1}^{*}A_{0}^{*}+\mathcal{Y}_{2n+1},
\label{cc4}
\end{equation}
 where $\mathcal{Y}_{m}$ is a Hermitian $2\times 2$ matrix given by a finite sum of products of exactly $m$ matrices from the lists
 \begin{align*}
  B_{0},\dots,B_{n-1};A_{0},\dots,A_{n-2};A_{0}^{*},\dots,A_{n-2}^{*}, &\quad\mbox{ if } m=2n,\\
  B_{0},\dots,B_{n-1};A_{0},\dots,A_{n-1};A_{0}^{*},\dots,A_{n-1}^{*}, &\quad\mbox{ if } m=2n+1.
 \end{align*}
Furthermore, for $m=2n+1$, each product in the sum for $\mathcal{Y}_{m}$ contains at least one term from the list $B_{0},\dots,B_{n-1}$. 
\end{lemma}
\begin{proof}
From the very definition of the matrix multiplication, we have
$$
\mathbb{P}_{0}\mathbb{J}^{m}\mathbb{P}_{0}^*=\sum_{j_{1}=0}^{\infty}\sum_{j_{2}=0}^{\infty}\dots\sum_{j_{m-2}=0}^{\infty}\sum_{j_{m-1}=0}^{\infty}\mathbb{J}_{0,j_{1}}\mathbb{J}_{j_{1},j_{2}}\cdots\mathbb{J}_{j_{m-2},j_{m-1}}\mathbb{J}_{j_{m-1},0}.
$$
Taking into account the tridiagonal structure of $\mathbb{J}$, i.e., $\mathbb{J}_{k_1,k_2}=0$ whenever $|k_1-k_2|>1$, this expression can be rewritten as
\begin{equation}
\mathbb{P}_{0}\mathbb{J}^{m}\mathbb{P}_{0}^*=\sum_{\mathbf j\in\mathcal{D}_{m}}\mathbb{J}_{0,j_{1}}\mathbb{J}_{j_{1},j_{2}}\cdots\mathbb{J}_{j_{m-2},j_{m-1}}\mathbb{J}_{j_{m-1},0},
 \label{eq.mth.moment.paths}
\end{equation}
where the finite multi-index set $\mathcal{D}_{m}$ is defined as
$$
 \mathcal{D}_{m}=\left\{\mathbf j=(j_{1},j_{2},\dots j_{m-1})\in\bbN_{0}^{m-1} \mid  (\forall i\in\{1,2,\dots,m\})(|j_{i}-j_{i-1}|\leq 1)\right\}
$$
and $j_{0}=j_{m}=0$. 

The multi-indices from $\mathcal{D}_{m}$ can be considered as labelings of certain paths well known in enumerative combinatorics~\cite{fla_dm80} (the Dyck path is a relevant but not exact term here). One can think about elements of $\mathcal{D}_{m}$  as of paths that start at the level~$0$, i.e., $j_{0}=0$, and then from a node $i-1$ can
\begin{itemize}
\item
go up:  $j_{i}=j_{i-1}+1$,
\item
go down: $j_{i}=j_{i-1}-1$, 
\item
stay horizontal: $j_{i}=j_{i-1}$. 
\end{itemize}
There are two more limitations on the paths: the end point has to be again at the level $0$, i.e., $j_{m}=0$, and the level of any node is always greater or equal to $0$, $j_{i}\geq0$ (paths cannot go to negative levels); see Figure~\ref{fig.path} for an illustration.

\begin{figure}[htb!]
    \centering
    \includegraphics[width=0.95\textwidth]{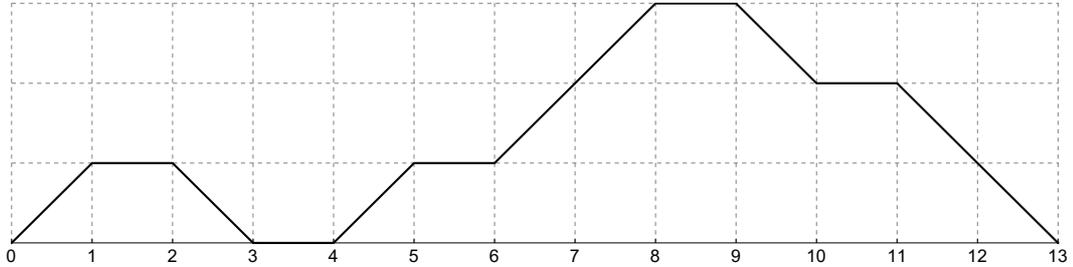}
    \caption{An example of the path corresponding to multi-index $(1,1,0,0,1,1,2,3,3,2,2,1)\in\mathcal{D}_{13}$; the respective term from the sum~\eqref{eq.mth.moment.paths} is the product $A_{0}B_{1}A_{0}^{*}B_{0}A_{0}B_{1}A_{1}A_{2}B_{3}A_{2}^{*}B_{2}A_{1}^{*}A_{0}^{*}$.}
    \label{fig.path}
\end{figure}

Let us first consider the even case $m=2n$. 
It is not difficult to realize that the only path in $\mathcal{D}_{2n}$ that reaches level $n$ is the path 
that goes from $0$ up to the middle and then down back to $0$; 
the corresponding multi-index $\mathbf j^{*}\in\mathcal{D}_{m}$ is
$$
 \mathbf j^{*}=(1,2,3,\dots,n-1,n,n-1,\dots,3,2,1), \quad\mbox{ if } m=2n,
$$
see Figure~\ref{fig.path-extreme} for an illustration. 
Recalling the notation $\mathbb{J}_{k,k}=B_{k}$ and $\mathbb{J}_{k,k+1}=\mathbb{J}_{k+1,k}^{*}=A_{k}$, we see that the sum in  \eqref{eq.mth.moment.paths} equals
$$
\mathbb{P}_{0}\mathbb{J}^{m}\mathbb{P}_{0}^*=
 A_{0}A_{1}\cdots A_{n-2}A_{n-1}A_{n-1}^{*}A_{n-2}^{*}\cdots A_{1}^{*}A_{0}^{*}+
 \sum_{\mathbf j\in\mathcal{D}_{2n}\setminus\{\mathbf j^{*}\}}\mathbb{J}_{0,j_{1}}\mathbb{J}_{j_{1},j_{2}}\cdots\mathbb{J}_{j_{2n-1},0}.
$$
Since no other paths (apart from $\mathbf j^{*}$) reach level $n$, the sum over indices $\mathcal{D}_{2n}\setminus\{\mathbf j^{*}\}$ contains only the products of $2n$ matrices from $B_{0},\dots,B_{n-1}$, $A_{0},\dots,A_{n-2}$, and $A_{0}^{*},\dots,A_{n-2}^{*}$, as claimed.

Next, consider the odd case $m=2n+1$. Here there are several paths reaching level $n$, but we single out the one going up to level $n$, having one horizontal step at level $n$ and then going down back to $0$: 
$$
 \mathbf j^{*}=(1,2,3,\dots,n-1,n,n,n-1,\dots,3,2,1), \quad\mbox{ if } m=2n+1,
$$
see Figure~\ref{fig.path-extreme}. 
\begin{figure}[htb!]
    \centering
    \includegraphics[width=0.95\textwidth]{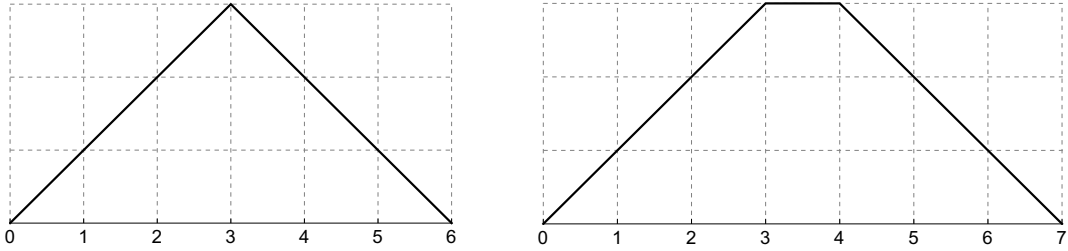}
    \caption{Examples of labelings $\mathbf j^{*}=(1,2,3,2,1)\in\mathcal{D}_{6}$ (left) and $\mathbf j^{*}=(1,2,3,3,2,1)\in\mathcal{D}_{7}$ (right); the corresponding terms from sum~\eqref{eq.mth.moment.paths} are $A_{0}A_{1}A_{2}A_{2}^{*}A_{1}^{*}A_{0}^{*}$ (left) and $A_{0}A_{1}A_{2}B_{3}A_{2}^{*}A_{1}^{*}A_{0}^{*}$ (right).}
    \label{fig.path-extreme}
\end{figure}
Now the sum in  \eqref{eq.mth.moment.paths} equals
$$
 A_{0}A_{1}\cdots A_{n-1}B_{n}A_{n-1}^{*}\cdots A_{1}^{*}A_{0}^{*}+
 \sum_{\mathbf j\in\mathcal{D}_{2n+1}\setminus\{\mathbf j^{*}\}}\mathbb{J}_{0,j_{1}}\mathbb{J}_{j_{1},j_{2}}\cdots\mathbb{J}_{j_{2n},0}.
$$
There are other paths in $\mathcal{D}_{2n+1}\setminus\{\mathbf j^{*}\}$  that reach level $n$. However, none of them contains a horizontal step at level $n$; in other words, the corresponding sum over the indices $\mathcal{D}_{2n+1}\setminus\{\mathbf j^{*}\}$ does not depend on $B_{n}$.  Moreover, since the total number of steps in any path is odd in this case, it is easy to see that every path will have at least one horizontal step, which means that every term in the sum over indices $\mathcal{D}_{2n+1}\setminus\{\mathbf j^{*}\}$ will contain at least one matrix $B_{0},\dots,B_{n-1}$.

To summarise: if we put
$$
\mathcal{Y}_{m}:= \sum_{\mathbf j\in\mathcal{D}_{m}\setminus\{\mathbf j^{*}\}}\mathbb{J}_{0,j_{1}}\mathbb{J}_{j_{1},j_{2}}\cdots\mathbb{J}_{j_{m-1},0},
$$
then the claims on the structure of $\mathcal{Y}_{m}$ follow. Finally, it is clear that $\mathcal{Y}_{m}$ is a~Hermitian matrix, for it is a difference of two Hermitian matrices $\mathbb{P}_{0}\mathbb{J}^{m}\mathbb{P}_{0}^*$ and the extracted term.
\end{proof}

\section{Uniqueness: proof of Theorem~\ref{thm.a3}(i)}\label{sec.c}

\subsection{Proof of uniqueness: plan}
Let $J$ be a bounded Jacobi matrix satisfying \eqref{a2} and let $\Lambda(J)=(\nu,\psi)$. 
Consider the self-adjoint block matrix $\bbJ$ with the Jacobi parameters \eqref{cc1}. 
Our plan of the proof of Theorem~\ref{thm.a3}(i) is as follows.
\begin{enumerate}[\rm (i)]
\item
Prove that the spectral data $(\nu,\psi)$ uniquely determine the $2\times2$ matrix spectral measure $\meas$ of $\bbJ$;
\item
Use Theorem~\ref{thm.cc2}(i) to infer that $\meas$ uniquely determines the equivalence class $[\bbJ]$;
\item
Prove that the equivalence class $[\bbJ]$ uniquely determines the coefficients $a_j$, $b_j$. 
\end{enumerate}
This logic can be represented by the diagram
$$
J\quad\rightarrow\quad
(\nu,\psi) \quad\rightarrow\quad
\meas \quad\rightarrow\quad
[\bbJ] \quad\rightarrow\quad
(A_j,B_j)\quad\rightarrow\quad
(a_j,b_j).
$$

\subsection{Step (i): $\meas$ is uniquely determined} 

Let us prove that the spectral data $\Lambda(J)=(\nu,\psi)$ uniquely determines the spectral measure $\meas$ of the self-adjoint block Jacobi matrix $\bbJ$. By \eqref{c1}, it suffices to check that the spectral data uniquely determines the $2\times2$ matrix $\bP_0f(\bJ)\bP_0^*$ for any continuous function $f$. By the Weierstrass approximation theorem, it suffices to check that the spectral data uniquely determines $\bP_0f(\bJ)\bP_0^*$ for any polynomial $f$. Finally, we see that it suffices to check that the spectral data uniquely determines the $2\times2$ matrix $\bP_0\bJ^m\bP_0^*$  for any $m\geq0$. 

Recalling \eqref{cc2}, \eqref{cc3}, we see that it suffices to check that the quantities $\omega_{2n}$, $\omega_{2n+1}$ are uniquely determined by $(\nu,\psi)$ for every $n\geq0$. But this follows directly from \eqref{a13a}, \eqref{a13b}: 
\begin{align}
\omega_{2n}
&=
\jap{(J^*J)^{n}\delta_0,\delta_0}
=
\int_0^\infty s^{2n}\dd\nu(s), 
\label{c2}
\\
\omega_{2n+1}
&=
\jap{J(J^{*}J)^{n}\delta_0,\delta_0}
=
\int_0^\infty s^{2n+1}\psi(s)\dd\nu(s).
\label{c3}
\end{align}
We have checked that the spectral data uniquely determines the spectral measure $\meas$ of $\bbJ$.

\subsection{Step (ii): $[\bbJ]$ is uniquely determined}
By Theorem~\ref{thm.cc2}(i),  the spectral measure $\meas$ uniquely determines the equivalence class $[\bbJ]$.

\subsection{Step (iii): the Jacobi parameters $a_j$, $b_j$ are uniquely determined}
It remains to show that only one matrix from the equivalence class $[\mathbb{J}]$ has the Jacobi parameters of the form
 $$
  A_{j}=\begin{pmatrix}
   0 & a_{j} \\ a_{j} & 0
  \end{pmatrix} \mbox{ with } a_j>0
  \quad\mbox{ and }\quad   
  B_{j}=\begin{pmatrix}
   0 & b_{j} \\ \overline{b}_{j} & 0
  \end{pmatrix},
 $$
for all $j\in\bbN_{0}$. To this end, we will use the following lemma whose proof is elementary.
 \begin{lemma}\label{lem.off-aig.unit.matr.id}
 Suppose
 $$
  \begin{pmatrix}
   0 & a \\ a & 0
  \end{pmatrix}U=
  \begin{pmatrix}
   0 & a' \\ a' & 0
  \end{pmatrix},
 $$
 where $a>0$, $a'>0$, and $U$ is a $2\times2$ unitary matrix. Then $a=a'$ and $U=I$.
 \end{lemma}
 Suppose $\mathbb{J}'$ is equivalent to $\mathbb{J}$, i.e., 
 \begin{equation}
\mathbb{J}'=\bbD(U)^{*}\mathbb{J}\bbD(U),
 \label{eq.equivalent.Jac.matr}
 \end{equation}
and the Jacobi parameters of $\mathbb{J}'$  are
$$
A_{j}'=\begin{pmatrix}
   0 & a_{j}' \\ a_{j}' & 0
  \end{pmatrix} \mbox{ with } a_j'>0
  \quad\mbox{ and }\quad   
  B_{j}'=\begin{pmatrix}
   0 & b_{j}' \\ \overline{b}_{j}' & 0
  \end{pmatrix},
 $$
for all $j\in\bbN_{0}$. Inspecting the first row of the matrices in~\eqref{eq.equivalent.Jac.matr}, we find
$$
 B_{0}'=B_{0} \quad\mbox{ and }\quad A_{0}'=U_{0}^{*}A_{0}U_{1}.
$$
Since $U_{0}=I$, it follows from Lemma~\ref{lem.off-aig.unit.matr.id} that $A_{0}'=A_{0}$ and $U_{1}=I$. Inspecting the second row in~\eqref{eq.equivalent.Jac.matr}, we get
$$
 B_{1}'=B_{1} \quad\mbox{ and }\quad A_{1}'=U_{1}^{*}A_{1}U_{2}.
$$
Since  $U_{1}=I$, it follows from Lemma~\ref{lem.off-aig.unit.matr.id} that $A_{1}'=A_{1}$ and $U_{2}=I$; and so on. 
Continuing this way, by induction we prove that $B_{j}'=B_{j}$ and $A_{j}'=A_{j}$ for all $j\geq0$. 
Thus, the equivalence class $[\bbJ]$ uniquely determines the Jacobi parameters $a_j>0$ and $b_j\in\bbC$. 
The proof of Theorem~\ref{thm.a3}(i) is complete.
\qed

\section{Surjectivity: proof of Theorem~\ref{thm.a3}(ii)}\label{sec.d}

\subsection{The set-up}
Let $\nu$ and $\psi$ be given; we need to construct the corresponding matrix $J$. Our plan is as follows:
\begin{enumerate}[\rm (i)]
\item
Define a $2\times2$ matrix valued measure $\meas$  by 
\begin{equation}
\int_{-\infty}^\infty (f_1(s)+sf_2(s))\dd\meas(s)
=
\int_0^\infty
\begin{pmatrix}
f_1(s)&\psi(s)sf_2(s)
\\
\overline{\psi(s)}sf_2(s)&f_1(s)
\end{pmatrix}
\dd\nu(s)
\label{d1}
\end{equation}
for any \emph{even} continuous functions $f_1$, $f_2$. 
\item
Prove that $\meas$ is non-degenerate in the sense of Theorem~\ref{thm.cc1}. 
\item
Use Theorem~\ref{thm.cc2}(ii) to assert the existence of a bounded block Jacobi matrix $\bbJ$ with the spectral measure $\meas$. 
\item
Prove that the equivalence class $[\bbJ]$ contains a matrix with the Jacobi parameters of the form
\begin{equation}
  A_{j}=\begin{pmatrix}
   0 & a_{j} \\ a_{j} & 0
  \end{pmatrix} \mbox{ with } a_j>0
  \quad\mbox{ and }\quad   
  B_{j}=\begin{pmatrix}
   0 & b_{j} \\ \overline{b}_{j} & 0
  \end{pmatrix}.
  \label{d2}
\end{equation}
\item
Consider the Jacobi matrix $J$ with the parameters $a_j$ and $b_j$ as above. 
Check that $J$ has the required spectral data, i.e. $\Lambda(J)=(\nu,\psi)$. 
\end{enumerate}
This logic can be represented by the diagram:
$$
(\nu,\psi) \quad\rightarrow\quad
\meas \quad\rightarrow\quad
[\bbJ] \quad\rightarrow\quad
(A_j,B_j) \quad\rightarrow\quad
(a_j,b_j) \quad\rightarrow\quad
J.
$$

\subsection{Step (i): definition of $\meas$}
Definition \eqref{d1} can be rewritten as follows: 
\begin{align}
\dd\meas(s)&=\frac12
\begin{pmatrix}
1&\psi(s)
\\
\overline{\psi(s)}&1
\end{pmatrix}
\dd\nu(s) \quad \text{ for $s\geq0$,}
\label{d3}
\\
\dd\meas(s)&=\frac12
\begin{pmatrix}
1&-\psi(-s)
\\
-\overline{\psi(-s)}&1
\end{pmatrix}
\dd\widetilde\nu(s) \quad \text{ for $s<0$,}
\label{d4}
\end{align}
where $\widetilde\nu$ is as in \eqref{a.sharp}.
The matrices in \eqref{d3}, \eqref{d4} are positive semidefinite. 
From here it is clear that $\meas$ is a non-negative matrix-valued measure. Since $\nu$ has bounded support, the same is true for $\meas$. 

\subsection{Step (ii): non-degeneracy of $\meas$}

The following lemma is crucial for our proof of surjectivity.
\begin{lemma}
Let $\meas$ be as in Step (i). Then $\meas$ is non-degenerate.
\end{lemma}
\begin{proof}
Let $P=\begin{pmatrix}p\\ q\end{pmatrix}$ be a $\bbC^{2}$-valued polynomial and suppose that $\|P\|_{L^2(\meas)}=0$, i.e. 
\begin{equation}
\int_{-\infty}^\infty \jap{\dd\meas(s)P(s),P(s)}_{\bbC^2}=0.
\label{d4a}
\end{equation}
We need to check that $P=0$. 
Let us fix a Borel representative of $\psi\in L^\infty(\nu)$ defined for all $s\geq0$ and satisfying $\abs{\psi(s)}\leq 1$ for \emph{all} $s\geq0$. With $\psi$ denoting this representative, we set 
\begin{align}
F(s):=&\,\frac12\Jap{\begin{pmatrix}1&\psi(s)\\\overline{\psi(s)}&1\end{pmatrix}\begin{pmatrix}p(s)\\ q(s)\end{pmatrix},\begin{pmatrix}p(s)\\ q(s)\end{pmatrix} }
\notag
\\
&+
\frac12\Jap{\begin{pmatrix}1&-\psi(s)\\-\overline{\psi(s)}&1\end{pmatrix}\begin{pmatrix}p(-s)\\ q(-s)\end{pmatrix},\begin{pmatrix}p(-s)\\ q(-s)\end{pmatrix}};
\label{d4b}
\end{align}
this is a bounded non-negative Borel function defined for all $s\geq0$. Let us define the Borel set
$$
S_0:=\{s\geq0: F(s)=0\}.
$$
By \eqref{d3}, \eqref{d4}, we can write \eqref{d4a} equivalently as
$$
\int_0^\infty F(s)\dd\nu(s)=0.
$$
From here it follows that $S_0$ has a full $\nu$-measure, i.e. $\nu(S_0)=1$.

Denote 
$$
S_1:=\{s\geq0: \abs{\psi(s)}=1\}
\quad\text{ and }\quad
S_2:=\{s\geq0: \abs{\psi(s)}<1\};
$$
these are disjoint Borel sets with $[0,\infty)=S_1\cup S_2$. 
The $2\times2$ matrix in \eqref{d3} has rank  one on $S_1$ and rank two on $S_2$. 

Let us split $S_0$ into a union of two disjoint Borel sets 
$$
S_0=(S_1\cap S_0)\cup (S_2\cap S_0).
$$
We claim that at least one of the two sets $S_1\cap S_0$ and $S_2\cap S_0$ is infinite. Indeed, if both of these sets are finite, then $S_0$ is also a finite set. But then we find that $\nu$ is supported on the finite set $S_0$, i.e. $\nu$ is a finite linear combination of point masses supported in $S_0$. This contradicts our assumption that the support of $\nu$ is infinite. 

Now consider two cases. 

\emph{Case 1: the set $S_2\cap S_0$ is infinite.}
Since both terms on the right hand side of \eqref{d4b} are non-negative, 
identity $F(s)=0$ for $s\in S_2\cap S_0$ implies that both of these terms vanish on this set. 
Let us focus on the first term for definiteness. Since the matrix in \eqref{d3} has rank two on $S_2$, it follows that $P(s)=0$ for all $s\in S_2\cap S_0$. As $p$, $q$ are polynomials and the set $S_2\cap S_0$ is infinite, we conclude that $P=0$, so in this case the proof is complete. 

\emph{Case 2: the set $S_1\cap S_0$ is infinite.}
Again, both terms on the right hand side of \eqref{d4b} vanish for $s\in S_1\cap S_0$. Now we need to use both terms. 
Using the fact that $\abs{\psi(s)}=1$ on $S_1\cap S_0$, we find  
$$
p(s)+\psi(s)q(s)=0
\quad\text{ and }\quad 
p(-s)-\psi(s)q(-s)=0
\quad\text{ for all $s\in S_1\cap S_0$. }
$$
Let $w$ be the polynomial
$$
w(s):=p(s)\overline{p}(-s)+q(s)\overline{q}(-s), \quad s\in\bbR.
$$
Let us stress that $\overline{p}$ (resp. $\overline{q}$) is the polynomial $p$ (resp. $q$) with complex conjugated coefficients.
Again using that $\abs{\psi(s)}=1$, we find
$$
w(s)=-\psi(s)q(s)\overline{\psi(s)}\overline{q(-s)}+q(s)\overline{q(-s)}=0
$$
for all $s\in S_1\cap S_0$. We conclude that the polynomial $w$ vanishes on the infinite set  $S_1\cap S_0$, hence $w=0$. It remains to check that this implies $P=0$. 

We would like to consider $w(s)$ for complex $s$. We have $\overline{p}(s)=\overline{p(\overline{s})}$ for $s\in\bbC$ and similarly for $q(s)$. Taking $s=\ii t$ with $t\in\bbR$, we write the definition of $w$ as
\begin{align*}
w(\ii t)&=p(\ii t)\overline{p(-\overline{\ii t})}+q(\ii t)\overline{q(-\overline{\ii t})}
\\
&=\abs{p(\ii t)}^2+\abs{q(\ii t)}^2=0
\end{align*}
for all $t\in\bbR$. It follows that $p(\ii t)=q(\ii t)=0$ for all $t\in\bbR$ and therefore $p=q=0$ as required. 
The proof is complete.
\end{proof}

\subsection{Step (iii): existence of $\bbJ$}
Since the measure $\meas$ is non-degenerate, by Theorem~\ref{thm.cc2}(ii) there exists a non-singular $2\times2$ block-Jacobi matrix $\bbJ$ with the spectral measure $\meas$.

\subsection{Step (iv): the equivalence class $[\bbJ]$ contains the ``correct'' representative}
Let $\bbJ$ be as defined at the previous step.
Here we show that the equivalence class $[\bbJ]$ contains a representative with the Jacobi parameters $A_j$, $B_j$ satisfying~\eqref{d2}. 
We use the moments of $\meas$, for which we find, by \eqref{d1}, taking $f_1(s)=s^{2n}$ and $f_2(s)=0$,
\begin{equation}
 \int_{-\infty}^\infty s^{2n}\dd\meas(s)=
 \begin{pmatrix}
  \omega_{2n} & 0 \\ 0 & \omega_{2n}
 \end{pmatrix},
 \quad
 \omega_{2n}=\int_0^\infty s^{2n}\dd\nu(s),
\label{eq.even.moments}
\end{equation}
and similarly taking $f_{1}(s)=0$ and $f_{2}(s)=s^{2n}$,
\begin{equation}
 \int_{-\infty}^\infty s^{2n+1}\dd\meas(s)=
 \begin{pmatrix}
  0 & \omega_{2n+1} \\ \overline{\omega}_{2n+1} & 0
 \end{pmatrix},  
\quad
\omega_{2n+1}=\int_0^\infty \psi(s)s^{2n+1}\dd\nu(s).
\label{eq.odd.moments}
\end{equation}
On the other hand, by \eqref{c1} the moments of $\meas$ can be expressed as
\begin{equation}
 \int_{-\infty}^\infty s^{m}\dd\meas(s)=\mathbb{P}_{0}\mathbb{J}^{m}\mathbb{P}_{0}^*, \quad m\geq0.
\label{eq.moment.ident}
\end{equation}
The proof proceeds by induction. It follows from the identity~\eqref{eq.moment.ident} with $m=1$ and~\eqref{eq.odd.moments} that
$$
 B_{0}=\mathbb{P}_{0}\mathbb{J}\mathbb{P}_{0}^*=
 \begin{pmatrix}
  0 & \omega_{1} \\ \overline{\omega}_{1} & 0
 \end{pmatrix}.
$$
Thus, $B_{0}$ already has the desired structure with $b_{0}:=\omega_{1}$. 
An inspection of the second moments in~\eqref{eq.moment.ident} and~\eqref{eq.even.moments} yields the equation
$$
 B_{0}^{2}+A_{0}A_{0}^{*}=\mathbb{P}_{0}\mathbb{J}^{2}\mathbb{P}_{0}^*=\begin{pmatrix}
  \omega_{2} & 0 \\ 0 & \omega_{2}
 \end{pmatrix}.
$$
By using the already obtained formula for $B_{0}$, we get 
$$
A_{0}A_{0}^{*}=\begin{pmatrix}
  \omega_{2}-|\omega_{1}|^{2} & 0 \\ 0 & \omega_{2}-|\omega_{1}|^{2}
 \end{pmatrix}.
$$
Since $A_0$ is non-singular (i.e. $\det A_0\not=0$) we have $\omega_{2}>|\omega_{1}|^{2}$; this can also be seen directly by  Cauchy--Schwarz. 
Hence we can choose the Jacobi parameter $A_{0}$ of the desired structure
$$
 A_{0}=\begin{pmatrix}
  0 & a_{0} \\ a_{0} & 0
 \end{pmatrix}, \quad\mbox{ with }\; a_{0}:=\sqrt{\omega_{2}-|\omega_{1}|^{2}}.
$$

Now, suppose that for some $n\in\bbN$, the matrices $B_{0},\dots,B_{n-1}$ and $A_{0},\dots,A_{n-1}$ have the off-diagonal structure~\eqref{d2}. We show that $B_{n}$ and $A_{n}$ can be chosen to be of the same structure. 

\emph{Step 1: $B_n$ has the required structure.} By~\eqref{eq.moment.ident}, \eqref{eq.odd.moments} and Lemma~\ref{lem.moments} we find
$$
 A_{0}A_{1}\dots A_{n-1}B_{n}A_{n-1}^{*}\dots A_{1}^{*}A_{0}^{*}+\mathcal{Y}_{2n+1}=\mathbb{P}_{0}\mathbb{J}^{2n+1}\mathbb{P}_{0}^*=
 \begin{pmatrix} 0 & \omega_{2n+1} \\ \overline{\omega}_{2n+1} & 0
 \end{pmatrix},
$$
where $\mathcal{Y}_{2n+1}$ is a Hermitian matrix given by a finite sum of products of exactly $2n+1$ matrices from the list 
$$
B_{0},\dots,B_{n-1};A_{0},\dots,A_{n-1};A_{0}^{*},\dots,A_{n-1}^{*}.
$$
By the induction hypothesis, all of these matrices have the structure \eqref{d2}, and in particular, they are all off-diagonal. A product of an odd number of $2\times 2$ off-diagonal matrices is off-diagonal. We also know that $ \mathcal{Y}_{2n+1}$ is Hermitian. It follows that 
$$
 \mathcal{Y}_{2n+1}= \begin{pmatrix} 0 & y_{2n+1} \\ \overline{y}_{2n+1} & 0
 \end{pmatrix},
$$
for some $y_{2n+1}\in\bbC$. 
From here we find that $B_{n}$ is a Hermitian matrix given by 
$$
 B_{n}=A_{n-1}^{-1}\dots A_{0}^{-1}
 \begin{pmatrix} 0 & \omega_{2n+1}-y_{2n+1} \\ \overline{\omega}_{2n+1}-\overline{y}_{2n+1} & 0
 \end{pmatrix}
 \left(A_{0}^{*}\right)^{-1}\dots ( A_{n-1}^{*} )^{-1}.
$$
Taking also into account that
$$
 A_{j}^{-1}=\begin{pmatrix} 0 & 1/a_{j} \\ 1/a_{j} & 0
 \end{pmatrix},\quad a_j>0, 
\quad j=0,1,\dots,n-1,
$$
we find that $B_{n}$ is also an off-diagonal Hermitian matrix, and so it has the desired structure \eqref{d2}.

\emph{Step 2:  $A_n$ has the required structure.} 
From~\eqref{eq.moment.ident}, \eqref{eq.even.moments} and Lemma~\ref{lem.moments} we find
$$
A_{0}\dots A_{n-1}A_{n}A_{n}^{*}A_{n-1}^{*}\dots A_{0}^{*}+\mathcal{Y}_{2n+2}=\mathbb{P}_{0}\mathbb{J}^{2n+2}\mathbb{P}_{0}^*=\begin{pmatrix} \omega_{2n+2} & 0 \\ 0 & \omega_{2n+2} \end{pmatrix}, 
$$
where $\mathcal{Y}_{2n+2}$ is a Hermitian matrix given by a finite sum of products of exactly $2n+2$ matrices from the list 
$$
B_{0},\dots,B_{n};A_{0},\dots,A_{n-1};A_{0}^{*},\dots,A_{n-1}^{*}.
$$
By the induction hypothesis and Step~1, all of these matrices have the structure \eqref{d2}, and in particular, they are all Hermitian and off-diagonal. It is easy to see that a product of an even number of $2\times 2$ Hermitian off-diagonal matrices is of the form
$$
\begin{pmatrix}
c&0\\0&\overline{c}
\end{pmatrix}, \quad c\in\bbC.
$$
We also know that $ \mathcal{Y}_{2n+2}$ is Hermitian. It follows that 
$$
\mathcal{Y}_{2n+2}=\begin{pmatrix} y_{2n+2} & 0 \\ 0 & y_{2n+2} \end{pmatrix},
$$
for some $y_{2n+2}\in\bbR$.
From here we find
\begin{align*}
A_{n}A_{n}^{*}&=A_{n-1}^{-1}\dots A_{0}^{-1}
\begin{pmatrix} 
\omega_{2n+2}-y_{2n+2}&0
\\
0&\omega_{2n+2}-y_{2n+2}
\end{pmatrix}
\left(A_{0}^{*}\right)^{-1}\dots ( A_{n-1}^{*} )^{-1}
\\
&=\frac{\omega_{2n+2}-y_{2n+2}}{|a_{0}|^{2}\dots|a_{n-1}|^{2}}
\begin{pmatrix} 1 & 0 \\ 0 & 1 \end{pmatrix}.
\end{align*}
Since all parameters $A_n$ are non-singular, we have  $\omega_{2n+2}>y_{2n+2}$. Now we can choose the Jacobi parameter $A_n$ as 
$$
 A_{n}=\begin{pmatrix}
  0 & a_{n} \\ a_{n} & 0
 \end{pmatrix}, \quad\mbox{ with }\; a_{n}:=\frac{\sqrt{\omega_{2n+2}-y_{2n+2}}}{|a_{0}|\dots|a_{n-1}|},
$$
which concludes the induction step. We have proved that the class of $\bbJ$ contains a~Jacobi matrix with the parameters of the form \eqref{d2}.

\subsection{Step (v): $J$ has the required spectral data}

Let $J$ be the Jacobi matrix with the parameters $a_j$, $b_j$ obtained at the previous step. All block Jacobi matrices in the class $[\bbJ]$ are bounded, i.e. the Jacobi parameters $A_j$, $B_j$ are bounded; it follows that $a_j$, $b_j$ are also bounded. 
Condition $a_j>0$ is satisfied by the previous step of the proof. 
Let $\Lambda(J)=(\nu_*,\psi_*)$ be the spectral data of $J$. We need to prove that $\nu_*=\nu$ and $\psi_*=\psi$. 

We use the moments of $\meas$. For these moments we have the expressions \eqref{eq.even.moments} and \eqref{eq.odd.moments}. On the other hand, by \eqref{cc2}, \eqref{cc3} and \eqref{c2}, \eqref{c3}, we find 
\begin{align*}
 \int_{-\infty}^\infty s^{2n}\dd\meas(s)&=
 \begin{pmatrix}
  \omega_{2n} & 0 \\ 0 & \omega_{2n}
 \end{pmatrix},
&
 \omega_{2n}&=\int_0^\infty s^{2n}\dd\nu_*(s),
\\
 \int_{-\infty}^\infty s^{2n+1}\dd\meas(s)&=
 \begin{pmatrix}
  0 & \omega_{2n+1} \\ \overline{\omega}_{2n+1} & 0
 \end{pmatrix},  
&
\omega_{2n+1}&=\int_0^\infty \psi_*(s)s^{2n+1}\dd\nu_*(s).
\end{align*}
It follows that 
\begin{align*}
\int_0^\infty s^{2n}\dd\nu(s)
&=
\int_0^\infty s^{2n}\dd\nu_*(s),
\\
\int_0^\infty \psi(s)s^{2n+1}\dd\nu(s)
&=\int_0^\infty \psi_*(s)s^{2n+1}\dd\nu_*(s)
\end{align*}
for all $n\geq0$. From the first equation we find that $\nu=\nu_*$ and from the second equation we find that $\psi(s)=\psi_*(s)$ for $\nu$-a.e. $s>0$. 
The proof of Theorem~\ref{thm.a3}(ii) is complete. \qed

\subsection{Modifications for complex $a_{j}$ with prescribed arguments}
\label{sec.d7}
One observes that the recursive algorithm for the reconstruction of the Jacobi parameters $A_{j}, B_{j}$ described in Step~(iv) determines uniquely the matrices $B_{j}$ and $A_{j}A_{j}^{*}$. Taking into account the off-diagonal structure of $A_{j}$ and $B_{j}$, see~\eqref{cc1}, the procedure determines uniquely $b_{j}$ and $|a_{j}|$. To reconstruct $a_{j}$ completely, we have used the assumed positivity of $a_{j}$.

Nevertheless, there is also the one-to-one correspondence between the set of Jacobi matrices whose Jacobi parameters satisfy
$$
b_{j}\in\bbC \quad\mbox{ and }\quad a_{j}\in\bbC\setminus\{0\}, \; \arg a_{j}=\theta_{j}, \quad\forall j\geq0,
$$
where $\theta_{j}\in(-\pi,\pi]$ is a fixed sequence, and the set of spectral data $(\nu,\psi)$ specified by Theorem~\ref{thm.a3}. The proof is essentially the same as the proof of Theorem~\ref{thm.a3}. Only a single step in both proofs of the uniqueness and surjectivity is to be slightly adjusted.

In Step~(iii) of the proof of Theorem~\ref{thm.a3}(i), one needs to show that the equivalence class $[\bbJ]$ contains only one block Jacobi matrix having the Jacobi parameters
 $$
  A_{j}=\begin{pmatrix}
   0 & a_{j} \\ a_{j} & 0
  \end{pmatrix} \mbox{ with } \arg a_j=\theta_{j}
  \quad\mbox{ and }\quad   
  B_{j}=\begin{pmatrix}
   0 & b_{j} \\ \overline{b}_{j} & 0
  \end{pmatrix}.
 $$
To this end, one uses the same inductive argument based on a modification of Lemma~\ref{lem.off-aig.unit.matr.id} that reads:
\begin{lemma*}
 Suppose
 $$
  \begin{pmatrix}
   0 & a \\ a & 0
  \end{pmatrix}U=
  \begin{pmatrix}
   0 & a' \\ a' & 0
  \end{pmatrix},
 $$
 where $a,a'\neq0$, $\arg a = \arg a'$, and $U$ is a $2\times2$ unitary matrix. Then $a=a'$ and $U=I$.
\end{lemma*}

The only modification in the proof of Theorem~\ref{thm.a3}(ii) occurs in Step~(iv). The inductive argument proceeds unchanged up to the point where parameter $a_{j}$ is to be determined by the equation
$$
 A_{j}A_{j}^{*}=\begin{pmatrix}
   r_{j} & 0 \\ 0 & r_{j}
  \end{pmatrix}, \quad r_{j}>0.
$$
Here we need to choose the solution 
$$
 A_{j}=\begin{pmatrix}
  0 & a_{j} \\ a_{j} & 0
 \end{pmatrix}, \quad\mbox{ with }\; a_{j}:=\sqrt{r_{j}}\,e^{\ii\theta_{j}}.
$$

\section{Proofs of Theorems~\ref{thm.a-homeo}, \ref{thm.a3a}, \ref{thm.a3b} and \ref{thm.og}} \label{sec.e}
\subsection{Proof of Theorem~\ref{thm.a-homeo}}
Let $J_N$ be a sequence of bounded Jacobi matrices satisfying \eqref{a2}, and suppose that $J_N\to J$ strongly as $N\to\infty$. Denoting $\Lambda(J_N)=(\nu_N,\psi_N)$ and $\Lambda(J)=(\nu,\psi)$, we need to prove that $(\nu_N,\psi_N)\to (\nu,\psi)$ weakly.

We start with the following remark: 
\begin{equation}
\norm{J}=\sup(\supp\nu).
\label{e3}
\end{equation}
Indeed, the norms of $J$ and $\abs{J}$ coincide. The norm of $\abs{J}$ can be expressed as the supremum of the support of the projection-valued spectral measure for $\abs{J}$. 
Since $\delta_0$ is the element of maximal type for $\abs{J}$, the support of the projection-valued spectral measure of $\abs{J}$ coincides with the support of $\nu$, and so the claim \eqref{e3} follows. 

Next, the norms of a strongly convergent sequence of operators are uniformly bounded, and therefore by \eqref{e3} the supports of $\nu_N$ are uniformly bounded. 

By complex conjugation, from the strong convergence $J_N\to J$ we obtain the strong convergence $J_N^*\to J^*$.
By \eqref{c2}, \eqref{c3}, we find
\begin{align}
\jap{(J_N^*J_N)^n\delta_0,\delta_0}
&=
\int_0^\infty s^{2n}\dd\nu_N(s), 
\label{e5}
\\
\jap{J_N(J_{N}^{*}J_N)^n\delta_0,\delta_0}
&=
\int_0^\infty s^{2n+1}\psi_N(s)\dd\nu_N(s)
\label{e6}
\end{align}
for all $N$. 
 It follows that we can pass to the limit in the left hand sides of~\eqref{e5} and~\eqref{e6}. After taking linear combinations over $n$, this yields
\begin{align*}
\int_0^\infty f(s^2)\dd\nu_N(s)&\to \int_0^\infty f(s^2) \dd\nu(s), 
\\
\int_0^\infty sf(s^2)\psi_N(s)\dd\nu_N(s)&\to \int_0^\infty sf(s^2) \psi(s)\dd\nu(s)
\end{align*}
as $N\to\infty$ for all polynomials $f$. Since the supports of $\nu_N$ are uniformly bounded, we can use the Weierstrass approximation theorem, which extends this convergence from polynomials $f$ to continuous functions $f$. This yields the weak convergence of spectral data. 

 Conversely, let $(\nu_N,\psi_N)$ be a sequence of spectral data convergent weakly to $(\nu,\psi)$. By Theorem~\ref{thm.a3}, we have $(\nu_N,\psi_N)=\Lambda(J_N)$ and $(\nu,\psi)=\Lambda(J)$ for unique bounded Jacobi matrices $J_N$, $J$ satisfying \eqref{a2}; we need to prove that $J_N\to J$ strongly. 

We need a lemma. For every $j\geq0$, let us denote by $a_j(J)$, $b_j(J)$ the Jacobi parameters of a bounded Jacobi matrix $J$ satisfying \eqref{a2}. 
\begin{lemma}
Let $J_N$, $J$ be bounded Jacobi matrices satisfying \eqref{a2} and $\Lambda(J_{N})\to\Lambda(J)$ weakly as $N\to\infty$. 
Then for every $j$, we have
$$
a_j(J_N)\to a_j(J), \quad b_j(J_N)\to b_j(J), \quad N\to\infty.
$$
\end{lemma}
\begin{proof}
Let us consider the $2\times 2$ block Jacobi matrix $\bbJ_N$ (resp. $\bbJ$) with the off-diagonal Jacobi parameters as in \eqref{d2}, corresponding to $J_N$ (resp. $J$). Furthermore, let $\meas_N$ (resp. $\meas$) be the $2\times2$ matrix valued spectral measure of $\bbJ_N$ (resp. $\bbJ$). By \eqref{eq.even.moments} and \eqref{eq.odd.moments} and by the weak convergence $(\nu_N,\psi_N)\to(\nu,\psi)$, the moments of $\meas_N$ converge to the corresponding moments of $\meas$:
$$
\int_{-\infty}^\infty s^m\dd\meas_N(s)\to\int_{-\infty}^\infty s^m\dd\meas(s), \quad N\to\infty, 
$$
for all $m\geq0$. By \eqref{eq.moment.ident}, it follows that 
\begin{equation}
\bbP_0 \bbJ_N^m \bbP_0^*\to \bbP_0 \bbJ^m \bbP_0^*, \quad N\to\infty,
\label{e7}
\end{equation}
for all $m\geq0$. 

The rest of the proof proceeds by induction. 
Since
$$
\bbP_0\bbJ\bbP_0^*=B_0
\quad\text{ and }\quad
\bbP_0\bbJ^2\bbP_0^*=B_0^2+A_0A_0^*,
$$
using \eqref{e7} with $m=1$ and $m=2$, we find that $b_0(J_N)\to b_0(J)$ and $a_0(J_N)\to a_0(J)$ as $N\to\infty$. 

Now suppose we have already established that 
$$
b_\ell(J_N)\to b_\ell(J), \quad a_\ell(J_N)\to a_\ell(J), \quad N\to\infty,
$$
for all $\ell=0,\dots,n-1$; let us prove this for $\ell=n$. We will rely on Lemma~\ref{lem.moments}. Denote by $A_n(\bbJ)$, $B_n(\bbJ)$ the Jacobi parameters of a block Jacobi matrix $\bbJ$. Our inductive assumption means that 
$$
B_\ell(\bbJ_N)\to B_\ell(\bbJ), \quad A_\ell(\bbJ_N)\to A_\ell(\bbJ), \quad N\to\infty,
$$
for all $\ell=0,\dots,n-1$. For the remainder terms from Lemma~\ref{lem.moments}, the induction hypothesis implies $\mathcal{Y}_{2n+1}(\bbJ_{N})\to\mathcal{Y}_{2n+1}(\bbJ)$ as $N\to\infty$ and so, bearing~\eqref{e7} in mind, the matrix
$$
 A_{0}(\bbJ_{N})\cdots A_{n-1}(\bbJ_{N})B_{n}(\bbJ_{N})A_{n-1}^{*}(\bbJ_{N})\cdots A_{0}^{*}(\bbJ_{N})
$$
converges to the matrix
$$
 A_{0}(\bbJ)\cdots A_{n-1}(\bbJ)B_{n}(\bbJ)A_{n-1}^{*}(\bbJ)\cdots A_{0}^{*}(\bbJ)
$$
for $N\to\infty$. Taking the induction hypothesis into account, we see that 
$$
A_{0}(\bbJ)\cdots A_{n-1}(\bbJ)B_{n}(\bbJ_{N})A_{n-1}^{*}(\bbJ)\cdots A_{0}^{*}(\bbJ)
$$
converges to 
$$
A_{0}(\bbJ)\cdots A_{n-1}(\bbJ)B_{n}(\bbJ)A_{n-1}^{*}(\bbJ)\cdots A_{0}^{*}(\bbJ)
$$
as $N\to\infty$. Since all matrices $A_{\ell}(\bbJ)$ are non-singular, we conclude that $B_n(\bbJ_N)\to B_{n}(\bbJ)$, which means that $b_n(J_N)\to b_n(J)$ as $N\to\infty$.

Similarly, using Lemma~\ref{lem.moments} with $m=2n+2$ in~\eqref{e7}, we find that $\mathcal{Y}_{2n+2}(\bbJ_{N})\to\mathcal{Y}_{2n+2}(\bbJ)$ and
$$
A_{0}(\bbJ_{N})\cdots A_{n}(\bbJ_{N})A_{n}^{*}(\bbJ_{N})\cdots A_{0}^{*}(\bbJ_{N})\to
A_{0}(\bbJ)\cdots A_{n}(\bbJ)A_{n}^{*}(\bbJ)\cdots A_{0}^{*}(\bbJ)
$$
as $N\to\infty$. From here, we deduce that 
$$
A_{n}(\bbJ_{N})A_{n}^{*}(\bbJ_{N})\to A_{n}(\bbJ)A_{n}^{*}(\bbJ)
$$
and so $a_n(J_N)\to a_n(J)$ as $N\to\infty$.
\end{proof}

We come back to the proof of Theorem~\ref{thm.a-homeo}.
By the lemma, we have 
$$
\norm{J_N\delta_j-J\delta_j}\to0, \quad N\to\infty,
$$
for every $j$. It follows that 
$$
\norm{J_N h-Jh}\to0, \quad N\to\infty
$$
for a dense set of elements $h\in\ell^2(\bbN_0)$. Finally, by the definition of the weak convergence of spectral data, the supports of $\nu_N$ are uniformly bounded, and therefore, by \eqref{e3}, the norms of $J_N$ are also uniformly bounded. It follows that $J_N\to J$ strongly as $N\to\infty$. 
The proof of Theorem~\ref{thm.a-homeo} is complete. \qed

\subsection{Proof of Theorem~\ref{thm.a3a}(i)}
We will use Proposition~\ref{prop.psi.strong.def}. 
First suppose that the spectrum of $J^*J$ is simple. Then by Theorem~\ref{thm.a1}, the element $\delta_0$ is cyclic for $\abs{J}$. 
It follows that the subspace $\calH_0$ coincides with the whole of $\ell^2(\bbN_0)$. By complex conjugation, the same is true for $\overline{\calH_0}$. Thus, the projection $\overline{\calP_0}$ in the left hand side of \eqref{eq.strong.psi} reduces to the identity operator. Now let us compute the norms of both sides of \eqref{eq.strong.psi} with $f=1$. We have
$$
\norm{J\delta_0}^2
=\jap{J^*J\delta_0,\delta_0}
=\int_0^\infty s^2\dd\nu(s)
$$
for the left hand side and similarly 
$$
\norm{\abs{J^*}\psi(\abs{J^*})\delta_0}^2
=
\int_0^\infty s^2\abs{\psi(s)}^2\dd\nu(s)
$$
for the right hand side. It follows that 
$$
\int_0^\infty s^2(1-\abs{\psi(s)}^2)\dd\nu(s)=0.
$$
Since the integrand is non-negative, we find that  $\abs{\psi(s)}=1$ for $\nu$-a.e. $s>0$. 
 
Conversely, suppose that  $\abs{\psi(s)}=1$ for $\nu$-a.e. $s>0$.
Again, consider the norms in \eqref{eq.strong.psi} for $f=1$. As above, we find
$$
\norm{\overline{\calP_0}J\delta_0}^2
=
\norm{\abs{J^*}\psi(\abs{J^*})\delta_0}^2 
=
\int_0^\infty s^2\dd\nu(s),
$$
and on the other hand, 
$$
\norm{J\delta_0}^2=\jap{J^*J\delta_0,\delta_0}=\int_0^\infty s^2\dd\nu(s).
$$
We conclude that 
$$
\norm{\overline{\calP_0}J\delta_0}=\norm{J\delta_0}
$$
and therefore, since $\overline{\calP_0}$ is an orthogonal projection, we find $J\delta_0\in\overline{\calH_0}$. By complex conjugation, we also have $J^*\delta_0\in\calH_0$.  By Lemma~\ref{lma.b1}, we conclude that $\calH_0=\ell^2(\bbN_0)$ and so $\delta_0$ is cyclic for $\abs{J}$. The proof is complete.

\subsection{Proof of Theorem~\ref{thm.a3a}(ii)}
We recall that the relation \eqref{a13b}:
\begin{equation}
\jap{J(J^{*}J)^n\delta_0,\delta_0}=\int_0^\infty s^{2n+1}\psi(s)\dd\nu(s)
\label{e1}
\end{equation}
for all $n\geq0$. Now suppose $J=J^*$; then the left hand side here is real, and therefore  
$$
\int_0^\infty s^{2n+1}\Im(\psi(s))\dd\nu(s)=0, \quad n\geq0.
$$
Taking linear combinations, from here we obtain that
$$
\int_0^\infty sg(s^2)\Im(\psi(s))\dd\nu(s)=0
$$
for all polynomials $g$ and by the Weierstrass approximation theorem, this is also true for all continuous functions $g$. 
From here we easily obtain that $\Im \psi(s)=0$ for $\nu$-a.e. $s>0$.

Conversely, suppose $\psi(s)$ is real for $\nu$-a.e. $s>0$. Then, taking into account formulas~\eqref{eq.even.moments}, \eqref{eq.odd.moments}, and \eqref{eq.moment.ident}, we see that the matrix $\mathbb{P}_{0}\mathbb{J}^{m}\mathbb{P}_{0}^{*}$ has real entries for all $m\geq0$. Now, with the aid of Lemma~\ref{lem.moments}, one easily proves by induction that $B_{j}$ is real for all $j\geq0$. It follows that $b_{j}\in\bbR$ for all $j\geq0$ and so $J$ is self-adjoint.

\subsection{Proof of Theorem~\ref{thm.a3a}(iii)}
Suppose $b_j=0$ for all $j\geq0$. Let $\Omega$ be the unitary operator of multiplication by the sequence $(-1)^j$ in $\ell^2(\bbN_0)$; in other words, $\Omega$ is the diagonal matrix with entries $1,-1,1,-1,\dots$ on the diagonal. Observe that our assumption $b_j=0$ implies the identity 
$$
\Omega J\Omega=-J. 
$$
Also notice that under our assumption $b_j=0$, we have $J=J^*$. Now for $n\geq0$ write
\begin{align*}
\jap{J^{2n+1}\delta_0,\delta_0}
&=
\jap{J^{2n+1}\Omega\delta_0,\Omega\delta_0}
=
\jap{\Omega J^{2n+1}\Omega\delta_0,\delta_0}
=
\jap{(\Omega J\Omega)^{2n+1}\delta_0,\delta_0}
\\
&=
\jap{(-J)^{2n+1}\delta_0,\delta_0}
=
-\jap{J^{2n+1}\delta_0,\delta_0}
\end{align*}
and so we conclude that $\jap{J^{2n+1}\delta_0,\delta_0}=0$ for all $n\geq0$. 
Recalling \eqref{e1}, we find that 
$$
\int_0^\infty s^{2n+1}\psi(s)\dd\nu(s)=0
$$
for all $n\geq0$. Similarly to the proof of the previous part, from here it is easy to conclude that $\psi(s)=0$ for $\nu$-a.e. $s>0$. 
 
Conversely, suppose $\psi(s)=0$ for $\nu$-a.e. $s>0$. By the previous part of the theorem, $J$ is self-adjoint, and by \eqref{e1} we obtain 
\begin{equation}
\jap{J^{2n+1}\delta_0,\delta_0}=0, \quad n\geq0. 
\label{e8}
\end{equation}
One can prove directly that this implies that $b_j=0$ for all $j\geq0$ (for example, $b_0=0$ follows directly from this identity with $n=0$). However, we prefer again to refer to Lemma~\ref{lem.moments}. By \eqref{cc3}, condition \eqref{e8} is equivalent to 
\begin{equation}
\bbP_0\bbJ^{2n+1}\bbP_0^*=0, \quad n\geq0.
\label{e9}
\end{equation}
Let us prove by induction that this implies $B_j=0$ for all $j\geq0$. As already noted, $B_0=0$ follows directly from \eqref{e9} with $n=0$. Assume that $B_0=\dots=B_{n-1}=0$. Let us use identity \eqref{cc4} of Lemma~\ref{lem.moments}. Recall that (this is the last part of the statement of Lemma~\ref{lem.moments}) each product in the sum for $\mathcal Y_{2n+1}$ contains at least one term from the list $B_0,\dots,B_{n-1}$. Thus, $\mathcal Y_{2n+1}=0$. Now it follows from \eqref{cc4} and \eqref{e9} that $B_n=0$. The proof is complete. 

\subsection{Proof of Theorem~\ref{thm.a3a}(iv)}
Follows readily from parts (i) and (ii).

\subsection{Proof of Theorem~\ref{thm.a3a}(v)}
Suppose $J$ is normal and the spectrum of $\abs{J}$ is simple. Normality implies $\abs{J^*}=\abs{J}$. 
As in the proof of part (i), we find $J\delta_0=g(\abs{J^*})\delta_0=g(\abs{J})\delta_0$, where $g(s)=s\psi(s)$. 
So we conclude that 
$$
J=\abs{J}\psi(\abs{J}), 
$$
as claimed.  
The proof of Theorem~\ref{thm.a3a} is complete. \qed

\subsection{Proof of Theorem~\ref{thm.a3b}}
For self-adjoint $J$, let us combine \eqref{a.moments} with \eqref{c2} and \eqref{c3}: 
\begin{align*}
\int_{-\infty}^\infty \lambda^{2n}\dd\mu(\lambda)
&=
\jap{J^{2n}\delta_0,\delta_0}
=
\int_0^\infty s^{2n}\dd\nu(s),
\\
\int_{-\infty}^\infty \lambda^{2n+1}\dd\mu(\lambda)
&=
\jap{J^{2n+1}\delta_0,\delta_0}
=
\int_0^\infty s^{2n+1}\psi(s)\dd\nu(s).
\end{align*}
From here we find
$$
\int_{-\infty}^\infty (f_1(\lambda)+\lambda f_2(\lambda))\dd\mu(\lambda)
=
\int_0^\infty (f_1(s)+sf_2(s)\psi(s))\dd\nu(s)
$$
for any continuous even functions $f_1$ and $f_2$. The required statement easily follows from here. \qed

\subsection{Proof of Theorem~\ref{thm.og}}
For each $j\geq0$, define the $2\times 2$ matrix polynomial
$$
  P_{j}:=\begin{pmatrix}
  q_{j}^{\rm{e}} & \overline{q_{j}^{\rm{o}}} \\[2pt]
  q_{j}^{\rm{o}} & \overline{q_{j}^{\rm{e}}}
  \end{pmatrix},
$$
where $q_{j}^{\rm{e}}$ and $q_{j}^{\rm{o}}$ denote the even and odd part of $q_{j}$, respectively.
The normalization $q_{0}=1$ implies that $P_{0}=I$. Using~\eqref{eq.q.recur}, one readily checks that 
$$
 sP_{j}(s)=P_{j-1}(s)A_{j-1}+P_{j}(s)B_{j}+P_{j+1}(s)A_{j},
$$
for all $j\geq0$, with matrices $A_{j}$ and $B_{j}$ as in~\eqref{cc1} and the convention $A_{-1}:=I$, $P_{-1}:=0$.  This is the three-term recurrence relation for the right matrix orthogonal polynomials corresponding to the block Jacobi matrix~\eqref{c.defblockJ}; see~\cite[Eq.~2.30]{dam-pus-sim_sat08} ($P_{j}$ coincides with $p_{j}^{R}$ in the notation of~\cite{dam-pus-sim_sat08}). Consequently, polynomials $P_{j}$ are orthonormal with respect to the spectral measure $\meas$ of $\mathbb{J}$, i.e.
\begin{equation}
 \int_{-\infty}^{\infty}P_{j}^{*}(s)\dd\meas(s)P_{k}(s)=\delta_{j,k}I, \quad j,k\geq0,
\label{eq.matrix.og.rel}
\end{equation}
see~\cite{dam-pus-sim_sat08} for details.

We know from~\eqref{d3} and~\eqref{d4} the relation between measure~$\meas$ and the spectral data $(\nu,\psi)$. It follows that we may rewrite the left-hand side of~\eqref{eq.matrix.og.rel} as
$$
\int_{0}^{\infty}\frac{1}{2}\left[P_{j}^{*}(s)\begin{pmatrix} 1 & \psi(s) \\ \overline{\psi(s)} & 1 \end{pmatrix} P_{k}(s)+P_{j}^{*}(-s)\begin{pmatrix} 1 & -\psi(s) \\ -\overline{\psi(s)} & 1 \end{pmatrix} P_{k}(-s)\right]\dd\nu(s).
$$
Using also the definition of $P_{j}$, a straightforward computation shows that the integrand equals
the matrix
$$
 \begin{pmatrix}
  Q_{j,k}(s) & 0 \\
  0 & Q_{k,j}(s)
 \end{pmatrix},
$$
where 
\begin{align*}
Q_{k,j}(s):=&\,\frac{1}{2}\left[q_{j}(s)\overline{q}_{k}(s)+q_{j}(-s)\overline{q}_{k}(-s)\right]\\
 &+\frac{1}{4}\left[\psi(s)+\overline{\psi(s)}\right]\left[q_{j}(s)\overline{q}_{k}(s)-q_{j}(-s)\overline{q}_{k}(-s)\right]\\
 &+\frac{1}{4}\left[\psi(s)-\overline{\psi(s)}\right]\left[q_{j}(s)\overline{q}_{k}(-s)-q_{j}(-s)\overline{q}_{k}(s)\right].
\end{align*}
Hence the matrix orthogonality relation~\eqref{eq.matrix.og.rel} reduces to the scalar form
$$
 \int_{0}^{\infty}Q_{k,j}(s)\dd\nu(s)=\delta_{j,k}, \quad j,k\geq0,
$$
which, when expressed in a vector form, yields the first orthogonality relation from the claim.

To deduce the second orthogonality relation, it suffices to introduce the unitary matrix 
$$
 U:=\frac{1}{\sqrt{2}}
 \begin{pmatrix}
 1 & -1 \\ 1 & 1 
 \end{pmatrix}
$$
and apply the identities 
$$
 U\begin{pmatrix}
 q_{j}(s) \\ q_{j}(-s)
 \end{pmatrix}
 =\sqrt{2}
 \begin{pmatrix}
 q_{j}^{\rm{o}}(s) \\ q_{j}^{\rm{e}}(s)
 \end{pmatrix}
$$
and
$$
 U\begin{pmatrix}
1+\Re \psi(s)& -\ii\Im\psi(s)
\\
\ii\Im\psi(s)&1-\Re\psi(s)
\end{pmatrix}U^{*}=
\begin{pmatrix}
 1 & \overline{\psi(s)} \\ \psi(s) & 1 
 \end{pmatrix}
$$
in the already proven orthogonality relation. The proof of Theorem~\ref{thm.og} is complete. \qed

\section{Example: Jacobi matrix with constant diagonals}\label{sec.example}

To illustrate our results in a concrete example, we compute the spectral data $(\nu_{\omega},\psi_{\omega})$ of the Jacobi matrix $J_{\omega}$ with constant Jacobi parameters
$$
 a_{j}=1 \quad\mbox{ and }\quad b_{j}=\omega,
$$
where $\omega\in\bbC$.

\subsection{Symmetries $\omega\leftrightarrow-\omega$}
First, we show that 
\begin{equation}
\nu_{-\omega}=\nu_{\omega} \quad\mbox{ and }\quad \psi_{-\omega}(s)=-\psi_{\omega}(s)
\label{eq:symm}
\end{equation}
for $\nu_{\omega}$-a.e. $s>0$.

Let $\Omega$ be the unitary operator of multiplication by the sequence $(-1)^j$ in $\ell^2(\bbN_0)$.
Then $\Omega J_{-\omega}\Omega=-J_{\omega}$. It follows that $\Omega J_{\omega}^{*}J_{\omega}\Omega=J_{-\omega}^{*}J_{-\omega}$. Equivalently, $\Omega|J_{\omega}|\Omega=|J_{-\omega}|$, and therefore 
$$
 \nu_{-\omega}(\Delta)=\langle\chi_{\Delta}(|J_{-\omega}|)\delta_{0},\delta_{0}\rangle=\langle\chi_{\Delta}(|J_{\omega}|)\Omega\delta_{0},\Omega\delta_{0}\rangle=\langle\chi_{\Delta}(|J_{\omega}|)\delta_{0},\delta_{0}\rangle=\nu_{\omega}(\Delta)
$$
for any Borel set $\Delta\subset\bbR$.

Next, observe that $J_{\omega}$ is normal. Then $|J_{\omega}|=|J_{\omega}^{*}|$ and one infers from~\eqref{a10} that
\begin{align*}
\int_{0}^{\infty}s\psi_{\omega}(s)f(s)\dd\nu_{\omega}(s)
&=
\langle J_{\omega}f(|J_{\omega}|)\delta_{0},\delta_{0}\rangle
=
-\langle J_{-\omega}f(|J_{-\omega}|)\Omega\delta_{0},\Omega\delta_{0}\rangle
\\
&=
-\langle J_{-\omega}f(|J_{-\omega}|)\delta_{0},\delta_{0}\rangle
=
-\int_{0}^{\infty}s\psi_{-\omega}(s)f(s)\dd\nu_{-\omega}(s)
\end{align*}
for all $\omega\in\bbC$ and continuous functions $f$. Since we already know that $\nu_{-\omega}=\nu_{\omega}$ we conclude $\psi_{-\omega}(s)=-\psi_{\omega}(s)$ for $\nu_{\omega}$-a.e. $s>0$.

\subsection{The spectral measure of $\bbJ_{\omega}$}
Let $\bbJ_{\omega}$ be the block Jacobi matrix~\eqref{c.defblockJ} with constant (i.e. $j$-independent) Jacobi parameters
$$
 A=\begin{pmatrix}0 & 1 \\ 1 & 0\end{pmatrix} \quad\mbox{ and }\quad B=\begin{pmatrix}0 & \omega \\ \overline{\omega} & 0\end{pmatrix}.
$$
In order to deduce explicit formulas for $\nu_{\omega}$ and $\psi_{\omega}$, we make use of the spectral measure $\meas_{\omega}$ of $\bbJ_{\omega}$. The latter is known to be expressible as the integral
$$
 \meas_{\omega}(\Delta)=\frac{1}{2\pi}\int_{-2}^{2}\sqrt{4-t^{2}}\,\chi_{\Delta}(At+B)\dd t,
$$
for any Borel set $\Delta\subset\bbR$, see~\cite[Theorem~3]{zyg_jpa01}. Diagonalizing the matrix
$$
 At+B=U(t)\begin{pmatrix} |t+\omega| & 0 \\ 0 & -|t+\omega| \end{pmatrix}U^{*}(t),
$$
with the unitary matrix
$$
 U(t)=\frac{1}{\sqrt{2}\,|t+\omega|}\begin{pmatrix}t+\omega & |t+\omega| \\ |t+\omega| & -t-\overline{\omega} \end{pmatrix},
$$
we find 
$$
\int_{-\infty}^{\infty}f(s)\dd\meas_{\omega}(s)=\frac{1}{2\pi}\int_{-2}^{2}\sqrt{4-t^{2}}\,U(t)\begin{pmatrix} f(|t+\omega|) & 0 \\ 0 & f(-|t+\omega|) \end{pmatrix}U^{*}(t)\dd t
$$
for all continuous functions $f$. Alternatively, the last equality rewrites as
\begin{align}
&\int_{-\infty}^{\infty}\left(f_{1}(s)+sf_{2}(s)\right)\dd\meas_{\omega}(s)\nonumber\\
&\hskip48pt =\frac{1}{2\pi}\int_{-2}^{2}\sqrt{4-t^{2}}\,\begin{pmatrix} f_{1}(|t+\omega|) & (t+\omega)f_{2}(|t+\omega| \\ (t+\overline{\omega})f_{2}(|t+\omega|) & f_{1}(|t+\omega|) \end{pmatrix}\dd t
\label{eq:int_f_d_M}
\end{align}
for any two even continuous functions $f_{1}$ and $f_{2}$.

We compare~\eqref{eq:int_f_d_M} with~\eqref{d1}. To this end, it is natural to substitute for $s=|t+\omega|$ in the integral on the right hand side of~\eqref{eq:int_f_d_M}. To do so, we have to distinguish two cases: $|\Re\omega|>2$ and $|\Re\omega|\leq2$. Due to the symmetries~\eqref{eq:symm} we can restrict the parameter $\omega$ to the half-plane $\Re\omega\geq0$ without loss of generality. 

\subsection{Case $\Re\omega>2$}
By squaring the substitution identity $s=|\omega+t|$, we get the quadratic equation
$$
 s^{2}=t^{2}+2t\Re\omega+|\omega|^{2},
$$
whose solutions are 
\begin{equation}
 t_{\pm}(s)=-\Re\omega\pm\sqrt{s^{2}-(\Im\omega)^2}.
\label{eq:def_t_plusminus}
\end{equation}
In the following, the expression under the square root is always non-negative and the square root assumes its principal branch. The correct choice of the sign is plus when $\Re\omega>2$ since, in this case, $t_{+}$ is an increasing function which maps the interval $[|\omega-2|,|\omega+2|]$ onto $[-2,2]$.

Then a direct calculation shows that the right hand side of \eqref{eq:int_f_d_M} coincides with the integral 
$$
\frac{1}{2\pi}\int_{|\omega-2|}^{|\omega+2|}\,\begin{pmatrix} f_{1}(s) & s\psi_{\omega}(s)f_{2}(s) \\ s\overline{\psi_{\omega}(s)}f_{2}(s) & f_{1}(s) \end{pmatrix}\dd\nu_{\omega}(s)
$$
for the absolutely continuous measure $\nu_{\omega}$ with the density
$$
\frac{\dd \nu_{\omega}}{\dd s}(s)=\frac{s}{2\pi}\sqrt{\frac{4-t_{+}^{2}(s)}{s^{2}-(\Im\omega)^2}}
$$
supported on the interval $[|\omega-2|,|\omega+2|]$, and 
$$
 \psi_{\omega}(s)=\frac{t_{+}(s)+\omega}{|t_{+}(s)+\omega|}=\frac{\sqrt{s^{2}-(\Im\omega)^2}+\ii\Im \omega}{s}.
$$

\subsection{Case $\Re\omega\leq 2$}

In this case, the line segment connecting complex numbers $\omega-2$ and $\omega+2$ intersects the imaginary line. When introducing the new variable $s=|t+\omega|$ into integral on the right of~\eqref{eq:int_f_d_M}, we integrate twice within a certain interval as $t$ varies from $-2$ to $2$. More concretely, assuming $0\leq\Re\omega\leq2$, as $t$ increases from $-2$ to $-\Re\omega$, the variable $s$ decreases from $|\omega-2|$ to $|\Im\omega|$. Further, as $t$ continues increasingly from $-\Re\omega$ to $2$, the variable $s$ increases from $|\Im\omega|$ to $|\omega+2|$. As a result, the integral in~\eqref{eq:int_f_d_M} splits into two integrals and the correct sign in $t_{\pm}$, as function of $s$, see~\eqref{eq:def_t_plusminus}, has to be chosen accordingly. 

Namely, bearing in mind the restriction $\Re\omega\geq0$, we find that the right hand side of \eqref{eq:int_f_d_M} equals
\begin{align*}
&\frac{1}{2\pi}\int_{|\Im\omega|}^{|\omega-2|}\sqrt{4-t^{2}_{-}(s)}\,
\begin{pmatrix} f_{1}(s) & (t_{-}(s)+\omega)f_{2}(s) 
\\ 
(t_{-}(s)+\overline{\omega})f_{2}(s) & f_{1}(s) \end{pmatrix}
\frac{s\dd s}{\sqrt{s^{2}-(\Im\omega)^2}}
\\
+&\frac{1}{2\pi}\int_{|\Im\omega|}^{|\omega+2|}\sqrt{4-t^{2}_{+}(s)}\,
\begin{pmatrix} f_{1}(s) & (t_{+}(s)+\omega)f_{2}(s) \\ (t_{+}(s)+\overline{\omega})f_{2}(s) & f_{1}(s) \end{pmatrix}
\frac{s\dd s}{\sqrt{s^{2}-(\Im\omega)^2}}.
\end{align*}
From this expression, when compared to~\eqref{d1}, one infers that measure $\nu_{\omega}$ is absolutely continuous with the density
$$
\frac{\dd \nu_{\omega}}{\dd s}(s)=\frac{s}{2\pi\sqrt{s^{2}-(\Im\omega)^2}}\begin{cases}
\!\sqrt{4-t_{-}^{2}(s)}+\sqrt{4-t_{+}^{2}(s)},  &\mbox{for } s\in[|\Im\omega|,|\omega-2|],\\[3pt]
\!\sqrt{4-t_{+}^{2}(s)},  &\mbox{for } s\in[|\omega-2|,|\omega+2|],
\end{cases}
$$
and
$$
\psi_{\omega}(s)=\frac{1}{s}\begin{cases}
\frac{(t_{-}(s)+\omega)\sqrt{4-t_{-}^{2}(s)}+(t_{+}(s)+\omega)\sqrt{4-t_{+}^{2}(s)}}{\sqrt{4-t_{-}^{2}(s)}+\sqrt{4-t_{+}^{2}(s)}},  &\mbox{ for } s\in(|\Im\omega|,|\omega-2|],\\[3pt]
t_{+}(s)+\omega,  &\mbox{ for } s\in[|\omega-2|,|\omega+2|].
\end{cases}
$$
By this formula $\psi_{\omega}$ is determined uniquely $\nu_{\omega}$-a.e., which means a.e. in $[|\Im\omega|, |\omega+2|]$. Recall that, if $0\in\supp\nu_{\omega}$, which is the case if and only if $\Im\omega=0$, we set $\psi_{\omega}(0)=1$ by the chosen normalization (here this normalization is unimportant since $0$ is never an atom of $\nu_{\omega}$).

\subsection{Remarks on special cases}

Observe that $\psi_{\omega}$ is unimodular $\nu_{\omega}$-a.e if and only if $|\Re\omega|\geq2$, hence the spectrum of $J_{\omega}^{*}J_{\omega}$ is simple exactly for these parameters by Theorem~\ref{thm.a3a}. If $|\Re\omega|<2$, $\delta_{0}$ is not a cyclic element of $J_{\omega}^{*}J_{\omega}$. The most degenerate case occurs when $\Re\omega=0$. In this case, we have
\begin{equation}
\frac{\dd \nu_{\omega}}{\dd s}(s)=\frac{s}{\pi}\sqrt{\frac{4-s^{2}-\omega^{2}}{s^{2}+\omega^{2}}}
\quad\mbox{ and }\quad
\psi_{\omega}(s)=\frac{\omega}{s}
\label{eq.nu.psi.im.om}
\end{equation}
for $s\in[|\omega|,|\omega-2|]$. 

On the other hand, if $\Im\omega=0$, then $J_{\omega}$ is self-adjoint and we have 
$$
\frac{\dd \nu_{\omega}}{\dd s}(s)=\frac{1}{2\pi}\sqrt{4-(s-\omega)^{2}}, \quad \psi_{\omega}(s)=1, \quad s\in[\omega-2,\omega+2], 
$$
if $\omega>2$, whereas
$$
\frac{\dd \nu_{\omega}}{\dd s}(s)=\frac{1}{2\pi}\begin{cases}
\sqrt{4-(s+\omega)^{2}}+\sqrt{4-(s-\omega)^{2}},  &\mbox{ for } s\in[0,2-\omega],\\[3pt]
\sqrt{4-(s-\omega)^{2}},  &\mbox{ for } s\in[2-\omega,2+\omega],
\end{cases}
$$
and
$$
\psi_{\omega}(s)=\begin{cases}
\frac{\sqrt{4-(s-\omega)^{2}}-\sqrt{4-(s+\omega)^{2}}}{\sqrt{4-(s-\omega)^{2}}+\sqrt{4-(s+\omega)^{2}}},  &\mbox{ for } s\in(0,2-\omega],\\[3pt]
1,  &\mbox{ for } s\in[2-\omega,2+\omega],
\end{cases}
$$
if $0\leq\omega\leq 2$. Further, it is easy to use standard methods to compute the spectral measure $\mu_{\omega}$ of $J_{\omega}$ when $\omega$ is real. The spectral measure $\mu_{\omega}$ is supported on $[\omega-2,\omega+2]$ and absolutely continuous with the density
$$
 \frac{\dd\mu_{\omega}}{\dd t}(t)=\frac{1}{2\pi}\sqrt{4-(t-\omega)^{2}}.
$$
If $\omega\geq2$, then $J_{\omega}\geq0$ and we see that $\nu_{\omega}=\mu_{\omega}$ and $\psi_{\omega}\equiv 1$, as expected. Using the notation introduced in~\eqref{a.sharp}, we have 
$$
 \frac{\dd\tilde{\mu}_{\omega}}{\dd t}(t)=\frac{1}{2\pi}\sqrt{4-(t+\omega)^{2}}
$$
for $t\in[-2-\omega,2-\omega]$. Then one readily checks using the obtained formulas that $\mu_{\omega}$ is related to $\nu_{\omega},\psi_{\omega}$ as claimed by Theorem~\ref{thm.a3b} also for $0\leq\omega<2$.

It is obvious that $\psi_{\omega}\equiv0$ for $\omega=0$. On the other hand, one easily sees that $\psi_{\omega}$ is non-zero if $\Im\omega\neq0$. Then one readily deduces from the obtained formulas that $\psi_{\omega}\equiv0$ only if $\omega=0$, which is in agreement with claim~(iii) of Theorem~\ref{thm.a3a}.

\subsection{Antilinear analogues to Chebyshev polynomials}

Consider polynomials $q_{j}$ given by recurrence~\eqref{eq.q.recur} with coefficients $a_{j}=1$, $b_{j}=\omega$ and $q_{0}=1$. If $J_{\omega}=J_{\omega}^{*}$, i.e. when $\Im\omega=0$, polynomials $q_{j}$ are nothing but the Chebyshev polynomials of the second kind. More precisely, if $\Im\omega=0$, we have 
$$
 q_{j}(s)=U_{j}\!\left(\frac{s-\omega}{2}\right), \quad j\geq0,
$$
which follows readily from the recursive definition of the Chebyshev polynomials of the second kind $U_{j}$ given by the formula
\begin{equation}
 U_{j+1}(x)=2xU_{j}(x)-U_{j-1}(x), \quad j\geq1, 
\label{eq.recur.U_j}
\end{equation}
and initial conditions $U_{0}(x)=1$ and $U_{1}(x)=2x$. 

With this observation, one may believe that there exists also a relation between $q_{j}$ and Chebyshev polynomials for general $\omega\in\bbC$. This is indeed the case at least when $\omega$ is purely imaginary; the authors are not aware of a closed formula for $q_{j}$ for general $\omega\in\bbC$.

If $\Re\omega=0$, then for all $j\geq0$, we have
\begin{equation}
 g_{2j+1}(s)=(-1)^{j}(s-\omega)\,U_{j}\!\left(\frac{2-s^{2}-\omega^{2}}{2}\right)
\label{eq.q_odd_iw}
\end{equation}
and 
\begin{align}
 g_{2j}(s)=(-1)^{j}\left[\frac{2}{2-s^{2}-\omega^{2}}\,U_{j}\!\left(\frac{2-s^{2}-\omega^{2}}{2}\right)-\frac{s^{2}+\omega^{2}}{2-s^{2}-\omega^{2}}\,T_{j}\!\left(\frac{2-s^{2}-\omega^{2}}{2}\right)\right],\nonumber\\
\label{eq.q_even_iw}
\end{align}
where $T_{j}$ denotes Chebyshev polynomials of the first kind. Recall that the polynomials $T_{j}$ satisfy the same recurrence~\eqref{eq.recur.U_j} as $U_{j}$ but with initial setting $T_{0}(x)=1$ and $T_{1}(x)=x$. Identities~\eqref{eq.q_odd_iw} and~\eqref{eq.q_even_iw} can be proven by induction $j$ with the aid of well-known properties of Chebyshev polynomials; we omit the details. 

When $\Re\omega=0$, we have~\eqref{eq.nu.psi.im.om} and therefore the general orthogonality relation from Theorem~\ref{thm.og} simplifies to the following integral identities 
\begin{align*}
\frac{1}{2\pi}\int_{|\omega|}^{|\omega-2|}\bigg\{&s\big[q_{j}(s)\overline{q}_{k}(s)+q_{j}(-s)\overline{q}_{k}(-s)\big]\\
&+\omega\big[q_{j}(s)\overline{q}_{k}(-s)-q_{j}(-s)\overline{q}_{k}(s)\big]\bigg\}\sqrt{\frac{4-s^{2}-\omega^{2}}{s^{2}+\omega^{2}}}\dd s=\delta_{j,k}
\end{align*}
for polynomials~\eqref{eq.q_odd_iw}, \eqref{eq.q_even_iw} and all $j,k\geq0$.

\subsection*{Acknowledgment}
F.~{\v S}. acknowledges the support of the EXPRO grant No.~20-17749X of the Czech Science Foundation and King's College London for hospitality in 2022.
A.P. is grateful to Czech Technical University in Prague for hospitality during two visits in 2022 and 2023.


\end{document}